\documentclass[11pt, notitlepage]{article}
\usepackage{amssymb,amsmath,comment}
\usepackage{epsfig,enumerate,amsmath,amsfonts,amssymb}
\usepackage{indentfirst}
\usepackage{pstricks}
\usepackage{setspace}
\usepackage{latexsym}
\usepackage{bm}
\usepackage{amsthm}
\usepackage{amsmath}
\usepackage{parskip}
\usepackage{fixmath}
\catcode`\@=11 \@addtoreset{equation}{section}
\def\thesection{\arabic{section}}

\def\theequation{\thesection.\arabic{equation}}
\catcode`\@=12
\usepackage{colortbl}
\usepackage{a4wide}

\newcommand{\ds} {\displaystyle}
\newcommand{\e}{\varepsilon}
\newcommand{\pa} {\partial}
\newcommand{\al} {\alpha}
\newcommand{\ba} {\beta}
\newcommand{\de} {\delta}
\newcommand{\ga} {\gamma}
\newcommand{\Om} {\Omega}
\newcommand{\ra} {\rightarrow}
\newcommand{\rp} {\rightharpoonup}
\newcommand{\De} {\Delta_{\mathbb{H}}}
\newcommand{\la} {\lambda}

\newcommand{\noi} {\noindent}
\newcommand{\na} {\nabla_{\mathbb{H}}}
\newcommand{\mb} {\mathbb}

\newcommand{\ld} {\langle}
\newcommand{\rd} {\rangle}
\newcommand{\ka}{\kappa}

\usepackage[all]{xy}
\catcode`\@=11
\def\theequation{\@arabic{\c@section}.\@arabic{\c@equation}}
\catcode`\@=12

\def\R{{I\!\!R}}
\def\H{{I\!\!H}}

\def\QED{\hfill {$\square$}\goodbreak \medskip}
\def\H{\mathbb{H}^N}
\def\R{\mathbb{R}^N}
\def\J{\mathcal{J}_a}
\newtheorem{Theorem}{Theorem}[section]
\newtheorem{Lemma}[Theorem]{Lemma}
\newtheorem{Proposition}[Theorem]{Proposition}

\begin{document}
\vspace{0.01in}

\title
{Brezis-Nirenberg type result for Kohn Laplacian  with critical Choquard  Nonlinearity }

\author{ {\bf Divya Goel\footnote{e-mail: divyagoel2511@gmail.com} \; 
	and \;  K. Sreenadh\footnote{
		e-mail: sreenadh@maths.iitd.ac.in}} \\ Department of Mathematics,\\ Indian Institute of Technology Delhi,\\
	Hauz Khaz, New Delhi-110016, India. }

\date{}

\maketitle

\begin{abstract}
In this article, we are study the following Dirichlet problem with Choquard type non linearity  
\[
-\Delta_{\mathbb{H}} u = a u+ \left(\int_{\Omega}\frac{|u(\eta)|^{Q^*_\lambda}}{|\eta^{-1}\xi|^{\lambda}}d\eta\right)|u|^{Q^*_\lambda-2}u  \; \text{in}\;
\Omega,\quad 
 u = 0 \;    \text{ on } \partial \Omega ,
\]
 where  $\Omega$ is a smooth bounded subset of the Heisenberg group $\mathbb{H}^N, N\in \mathbb N$  with $C^2$ boundary and $\Delta_{\mathbb{H}}$ is the Kohn Laplacian on the Heisenberg group $\mathbb{H}^N$. Here,  $Q^*_\lambda=\frac{2Q-\lambda}{Q-2},\; Q= 2N+2$
and $a$ is a positive real parameter. We derive  the 
Brezis-Nirenberg type result for the  above problem.  Moreover, we also prove the regularity of solutions and nonexistence of solutions depending on the range of $a$.   
\medskip
  
\noi \textbf{Key words:} Kohn Laplacian, Heisenberg group, Hardy-Littlewood-Sobolev inequality, critical exponent.

\medskip

\noi \textit{2010 Mathematics Subject Classification: 35H20, 43A80, 35B33, 35A15.}

\end{abstract}
\section{Introduction}
The purpose of this article is to  establish the Brezis-Nirenberg type result for the following semi linear Dirichlet homogeneous boundary value problem:  
\begin{equation*}
(P_a)\;
\left\{
-\De u = a u+ \left(\int_{\Om}\frac{|u(\eta)|^{Q^*_\la}}{|\eta^{-1}\xi|^{\la}}d\eta\right)|u|^{Q^*_\la-2}u\;\;  \text{in } \Omega,\;\;
u=0 \; \text{ on } \pa \Om,
\right.
\end{equation*}
\noi where $\Om$ is a smooth bounded subset of the Heisenberg group $\H, N\in \mathbb N$  with $C^2$ boundary and $\De$ is the subelliptic Laplacian on $\H$.  Here   $0<\la<N$,  $a\in \mathbb R$, $Q=2N+2$  and $ Q^*_\la=\frac{2Q-\la}{Q-2}$ is the critical exponent in the sense of Hardy-Littlewood-Sobolev inequality \eqref{hg1}. We recall that  Choquard equation \eqref{co20}  was first introduced in  the pioneering work of H.~Fr\"ohlich \cite{froh} and S.~Pekar \cite{pekar} for the  modeling of  quantum polaron:
\begin{equation}\label{co20}
-\Delta u +u = \left(\frac{1}{|x|}* |u|^2\right) u \text{ in } \mathbb{R}^3.
\end{equation}
As pointed out by Fr\"ohlich \cite{froh} and Pekar, this model corresponds to the study of  free electrons
in an ionic lattice interact with phonons associated to deformations of the
lattice or with the polarization that it creates on the medium (interaction
of an electron with its own hole). 
Due to its applications in Physical models, later many people studied groundstate solutions of  Choquard equations (See \cite{moroz2,moroz3,ehleib}). Recently, Yang and Gao(\cite{yang,yangjmaa}) studied the boundary value problem for Choquard equation. In particular, they considered the following equation 
\begin{equation*}
-\Delta u = \la h(u)+ \left(\int_{\Om}\frac{|u(y)|^{2^*_{\mu}}}{|x-y|^{\mu}}dy\right)|u|^{2^*_{\mu}-2}u \text{ in } \Om, \quad
u=0 \text{ on } \pa \Om, 
\end{equation*}
\noi where  $\Om\subset \mb R^N, N\ge 3$ is a bounded domain having smooth boundary $\pa \Om$,  $\la>0$, $0<\mu<N$. For more details on recent works on Choquard equation we refer to \cite{guide,ts} and the reference therein. \\
  For the sake of reader convenience, we will  give a short brief on the Heisenberg group and Kohn  Laplacian. The Heisenberg Group $\H = \mathbb{R}^N\times\mathbb{R}^N\times \mathbb{R},\; N\in \mathbb{N}$ is a  Lie group, endowed with the  following group law 
\begin{align*}
(x,y,t)\cdot (x^\prime,y^\prime,t^\prime	)= (x+x^\prime,y+y^\prime, t+t^\prime+2(\ld x,y^\prime\rd - \ld x^\prime, y \rd  )), 
\end{align*}
where $x,y, x^\prime, y^\prime \in \mathbb{R}^N$. The corresponding Lie algebra of left invariant vector fields  is generated by  the following vector  fields
$$
 X_j= \frac{\pa}{\pa x_j}+ 2y_j\frac{\pa}{\pa t},  \;
Y_j= \frac{\pa}{\pa y_j}- 2x_j\frac{\pa}{\pa t} , \;
T= \frac{\pa}{\pa t}.$$
It is straightforward to check that for all $ j,k =1,2,\cdots, N,$
\begin{equation}\label{hg33}
 [X_j,X_k]=[Y_j,Y_k]= [X_j, \frac{\pa}{\pa t}]= [Y_j, \frac{\pa}{\pa t}]=0, 
  \text{ and } [X_j, Y_k]= -4\de_{jk}\frac{\pa}{\pa t}.
\end{equation}
 These relations \eqref{hg33} establish the Heisenberg's canonical commutation relations of quantum mechanics for position and momentum, hence the name Heisenberg group \cite{hei}.  The second order self-adjoint operator,  $\De$ on $\H$  is defined as follows:
\begin{align*}
\De= \sum_{j=1}^{N}X_j^2+Y_j^2, 
\end{align*}
is usually called  subelliptic Laplacian or Kohn Laplacian on $\H$. The foundational work of H\"ormander\cite{hor} on the operators of type sum of squares of vector fields attracted lot of researchers and as a result, there is a considerable amount of  development in this field related to analysis of homogeneous Lie groups. The operator  $\De$ is hypoelliptic and the fundamental solution of this operator was given by Folland \cite{folland2}. Later in \cite{folland1}, Folland also proves some subelliptic estimates and provides the function spaces on the nilpotent Lie groups. The divergence form of Kohn Laplacian is defined  as   $\De u= \nabla \cdot (A\nabla u )$, where $A$ is the following $(2N+1)\times(2N+1)$ matrix:
\begin{equation}\label{hg25}
\begin{aligned}
\begin{bmatrix}
I & 0 & 2y^T  \\
0 & I & 2x^T \\
2y & -2x& 4(x^2+y^2)\\
\end{bmatrix} ,
\end{aligned}
\end{equation}
where $I$ is $(N\times N)$ identity matrix and $x^2+y^2= \ds \sum_{j=1}^{N} x_j^2+y_j^2$. Hence the following Guass-Green formula holds:
\begin{align*}
\int_{\Om}\De u~v d\xi= -\int_{\Om}\na u\cdot \na v \; d\xi + \int_{\pa \Om} v A\nabla u . \nu \; d\sigma, 
\end{align*}
where $\na u$ is the $2N$- vector $(X_1u,X_2u,\cdots, X_Nu, Y_1 u, Y_2u, \cdots Y_Nu)$  and $\nu$ is unit outward normal to the boundary $\pa \Om$. We define the left translations on $\H$ by 
\begin{align*}
\tau_\xi:\H \ra \H \qquad \tau_\xi(\xi^\prime)= \xi \cdot\xi^\prime,
\end{align*}
and the  natural $\mathbb{H}$-dilations $\de_\theta: \H \ra \H $ by 
$$
 \de_\theta(x,y,t)= (\theta x, \theta y, \theta^2 t), 
$$
for $\theta>0$. It is easy to verify that the $\De$ is invariant with respect to left translations and  it is homogeneous of degree $2$ with respect to dilation $\de_\theta$. That is, 
\begin{align*}
\De(u\; o\; \tau_\xi)= \De u \; o\; \tau_\xi,\qquad \De(u\; o\; \de_\theta)= \theta^2 \De u \; o\; \de_\theta. 
\end{align*}
The Jacobian determinant of $\de_\theta$ is $\theta^Q$. The number $Q=2N+2$ is called the homogeneous dimension of $\H$ and it portrays a role equivalent to the topological dimension in the Euclidean space.  We denote the homogeneous norm on $\H$ by  
\begin{align*}
|\xi|= |(x,y,t)|= (t^2+ (x^2+y^2)^2)^{1/4}, \text{ for all } \xi = (x,y,t) \in \H.
\end{align*}
We shall denote $B(\xi,r)$, the ball of center $\xi$ and radius $r$. It implies $\tau_\xi(B(0,r))= B(\xi,r)$ and $\de_r(B(0,1))= B(0,r)$. 
We denote $\varGamma^2(\overline{\Om})$ be the space of all continuous functions $u$ on $ \overline{\Om}$  such that $X_ju,\; Y_ju,\; X_j^2u$ and $Y_j^2u$ are all continuous in $\Om$ which  can be continuously extended up to the boundary of $\Om$. Analogous to space $H^{1}(\R)$, Folland and Stein \cite{follandstein} introduced the space $S_1^2(\H)$ which is related to Vector fields $X_j$ and  $Y_j$. The space 
\begin{align*}
S_1^2(\H)= \bigg\{u \in L^2(\H)\; :\; X_ju,Y_ju \ \in L^2(\H) \text{ for all } j=1,2,\cdots N\bigg\}, 
\end{align*}
 is a Hilbert space with  inner product 
 \begin{align*}
\langle u, v\rangle=\int_{\H} \na u \cdot \na v ~ d\xi + \int_{\H}  u   v ~ d\xi 
 \end{align*}
 and the corresponding norm is
 \begin{align*}
\|u\|^2_{S_1^2(\H)} = \int_{\H}|\na u |^2 ~d\xi+ \int_{\H} |u|^2~ d\xi.
\end{align*}
 The space $\mathring{S}_1^2(\Om)$ is the closure of $C_c^{\infty}(\Om)$ in $S_1^2(\H)$. Then  a Poincare type inequality shows that $\mathring{S}_1^2(\Om)$ is a Hilbert space with the norm 
\begin{align*}
\|u\|^2_{\mathring{S}_1^2(\Om)} = \int_{\Om}|\na u |^2 ~d\xi.
\end{align*}
In \cite{follandstein}, Folland and Stein proved the following Sobolev type inequality:
There exists a positive constant $C_Q$ such that 
\begin{align}\label{hg2}
|u|_{Q^*}\leq C_Q \|u\|_{S_1^2(\H)},\; \text{for all} \; u\in S_{1}^{2}(\H).
\end{align}
where $|\cdot |_{Q^*}$ is the norm in $L^{Q^*}(\H)$ and $Q^*= \frac{2Q}{Q-2}$. Moreover, when $\Om$ is bounded then $\mathring{S}_1^2(\Om)$ is continuously embedded in $L^p(\Om)$ for all $p\in [1, Q^*]$. Moreover, these embedding are compact when $p<Q^*$ and the embedding is not compact in the limiting case $p=Q^*$. Next, let $\{a_n\}$ be the sequence of eigenvalues of $-\De $ on $\Om$ with zero Dirichlet boundary data.  Then  $a_n\ra \infty$ as $  n \ra \infty$ and
\begin{align*}
0< a_1\leq a_2\leq a_3\leq \cdots\leq a_n \leq \cdots
\end{align*}
 Also,  the variational characterization of $a_1$ is given by 
\begin{align*}
a_1= \inf_{ u \in \mathring{S}_1^2(\Om)\setminus \{0\}}  \bigg \{   \int_{\Om}|\na u|^2~dx :\; \int_{\Om}| u|^2~dx=1 \bigg\}.
\end{align*}
Moreover, let $\{e_n\} \subset \mathring{S}_1^2(\Om)$ be the sequence of eigen functions corresponding to eigenvalues $\{a_n\}$. 
If we denote 
\begin{align*}
\mathbb{Y}_n = \text{ span}\{ e_1, e_2, \cdots, e_n  \} \text{ and }\mathbb{E}_{n+1} := \{ u \in \mathring{S}_1^2(\Om)\; : \; \ld u, e_j\rd =0,  \text{ for all } j=1,2,\cdots, n     \}, 
\end{align*}
then $\mathring{S}_1^2(\Om)= \mathbb{Y}_n \oplus \mathbb{E}_{n+1}$ for all $n \in \mathbb{N}$. By using standard Di girogi technique,  one can have $e_n \in L^\infty(\Om)$ for all $n \in \mathbb{N}$. 
The best constant for the embedding $S_1^2(\H)$ into  $L^{Q^*}(\H)$ is defined as 
\[S=\inf_{ u \in S_1^2(\H) \setminus \{0\}} \left\{ \int_{\H} |\na u|^2 dx:\; \int_{\H}|u|^{Q^{*}}dx=1\right\}.\]
By \eqref{hg2}, $S>0$. Also, for any non-empty open set $\Om \subset \H$, we have
\begin{align}\label{hg3}
S=\inf_{ u \in \mathring{S}_1^2(\Om) \setminus \{0\}} \left\{ \int_{\H} |\na u|^2 dx:\; \int_{\H}|u|^{Q^{*}}dx=1\right\}.
\end{align}
That is, infimum  is achieved  when $\Om = \H$. In the pioneering work, Jersion and Lee \cite{jl1,jl2}  proved that, up to a positive constant $B_0$, the function 
\begin{align}\label{hg4}
Z(\xi)=Z(x,y,t)= \frac{B_0}{(t^2+(1+|x|^2+|y|^2)^2)^{(Q-2)/4}}
\end{align}
is such that $S^{Q/2}= \|Z\|_{S_1^2(\H)}^2= |Z|_{Q^*}^{Q^*}$. Further, any minimizer of $S$  takes the form  
\[
z_{\ba \omega}(\xi)= \ba^{\frac{Q-2}{2}} Z(\de_\ba(\omega^{-1}\xi)), 
\]
for suitably $\ba >0$ and $\omega \in \H$.

The inspiring point for studying the  semilinear equations on Heisenberg group is the fact that these equations  arise as Euler-Lagrange equations in
some variational problems on Cauchy--Riemann (CR) manifolds. For more details, we refer to the works of Jerison and Lee \cite{jl1,jl2} on the CR
Yamabe problem.  The Dirichlet problem on Kohn Laplacian was first studied by Jerison\cite{jerison1,jerison2}. After that there were many article on the Dirichlet problem on Kohn Laplacian. For instance, in \cite{citti}, Citti proved the existence of a positive solution for the following equation 
\[
-\De u + au = u^{Q^*-1}+ f(x, u)=0,\text{ in } D  \quad
u= 0 \text{ on  } \pa D, 
\]
where $f$ is a lower order term that is,   $f(x, u)= o(|u|^{Q^*-1})$ as $|u| \ra \infty$ and $a \in L^\infty(\Om)$.  For more works on Dirichlet problem for Kohn Laplacian we refer the readers to  \cite{biagini,birini,birindelli,garofalo1,garofalo2} and references therein. 

Recently  a great attention has been focused on nonlocal equations on Euclidean domain, both for
the pure mathematical research and in view of concrete real-world applications. There is a substantial amount of article which discusses the existence, nonexistence and multiplicity of solutions. For detailed study one can go through \cite{servadei, yang,yangjmaa} and references therein. 
We need to point out that  all the works on Heisenberg group  mentioned above talks about the semilinear Dirichlet problem involving   Sobolev critical exponent (that is, $Q^*$) on Heisenberg group.
But no article talks about the nonlocal equations of Choquard type on Heisenberg Group. Since nonlocal critical equations are relevant for their relations with problems arising in differential geometry, where a lack of compactness occurs therefore, it is essential to study the nonlocal Dirichlet problem  on Heisenberg group. 

Motivated by all above, in this article we consider the nonlocal Dirichlet problem $(P_a)$ with Choquard type critical nonlinearity. 
First question that is natural to ask for $(P_a)$ is  the existence of weak solution in the lines of renowned result of  Brezis and Nirenberg \cite{brezis}.  In this article we answer  this question affirmatively. 
The salient feature of this article is the blow-up analysis  to study the critical level and compactness of Palais-Smale sequences.
  The estimates on the critical term are delicate and uses various inequalities involving the minimizers. 
  Furthermore, we prove the  regularity of solutions to $(P_a)$. We also prove Pohozaev type identity(See Lemma \ref{lemhg14}) for convolution type nonlinearity for Kohn Laplacian which itself a novel  result and of autonomous interest and subsequently give a nonexistence result in star shaped domain.  We highlight that no result is available in the current literature for Choquard equation on Heisenberg group.
  In this regard, the results proved in the present article are completely new.
    The main results proved in this article are the following.
\begin{Theorem}\label{thmhg1} 	Let $\Om$ be a bounded domain of $\H$. Then  
\begin{enumerate} \item[(i)] for all  $0< \la<Q$ and $a>0$,  $a\ne a_k, k=1,2,...$,   $(P_a)$ has a nontrivial solution. 
\item[(ii)] any weak solution $u \in \mathring{S}_1^2(\Om)$ of problem $(P_a)$ belongs to $\Gamma^{2+\al}(\Om) \cap \Gamma^{\al_1}(\overline{\Om})$ for some $\al, \al_1  \in (0,1)$. 
\end{enumerate}
\end{Theorem} 
Next we prove the following nonexistence result.
\begin{Theorem}\label{thmhg2}
	Let $a<0$ and $\Om $ is a smooth bounded domain in $\H$, which is strictly star-shaped with respect to the origin in $\H$, then any solution $u \in \Gamma^2(\overline{\Om})$ of $(P_a)$ is trivial. 
\end{Theorem}
Turing to layout of the paper, in Section 2, we give the some notations and preliminary results.  For the sake of Clarity, we divide the proof of Theorem \ref{thmhg1}(i)  in two cases. In Section 3, we discuss the case  $0<a<a_1$ for Theorem \ref{thmhg1}. In Section 4, we discuss the case $a \geq a_1$ for Theorem \ref{thmhg1}(i).  In Section 5, we give the regularity   results and nonexistence result that is proof Theorem \ref{thmhg1}(ii)  and Theorem \ref{thmhg2}. 
\section{Preliminaries and Notation}
The beginning of  variational  approach for the Problem $(P_a)$ is the following proposition which originates from the work of Folland and Stein \cite{follandstein} and recently Leib and Frank \cite{leibfrank} proved the existence of maximizers for the equality in \eqref{hg1}. 
\begin{Proposition}\label{prophg1}
	Let $r,s>1$ and $0<\la <Q$ with $1/r+\la/Q+1/s=2$, $f\in L^r(\H)$ and $h\in L^s(\H)$. There exists a sharp constant $C(t,r,\la,Q)$ independent of $f,h$, such that 
	\begin{equation}\label{hg1}
	\int_{\H}\int_{\H}\frac{f(\xi)h(\eta)}{|\eta^{-1}\xi|^{\la}}~d\eta d\xi \leq C(t,r,\la,Q) |f|_{r} |h|_{s}.
	\end{equation}
	If $r=s=2Q/(2Q-\la)$, then 
	\begin{align*}
	C(t,r,\la,Q)=C(Q,\la)=   \left(\frac{\pi^{N+1}}{2^{N-1}N!}\right)^{\la/Q}\frac{N!\Gamma((Q-\la)/2)}{\Gamma^2((2Q-\la)/2)},
	\end{align*}
where $\Gamma$ is the usual Gamma function.	Equality holds in  \eqref{hg1} if and only if $f\equiv (constant)h$ and 
	\begin{align*}
	h(\xi)= cU(\de_\theta(a^{-1}\xi))
	\end{align*} 
	for some $c \in \mathbb{C},\;  \theta>0,\; a \in \H (\text{unless } f\equiv 0 \text{ or }g\equiv 0)$ and $U$ is defined as
	\begin{align}\label{hg34}
	U(\xi)= U(x,y,t)= (t^2+(1+|x|^2+|y|^2)^2)^{-(2Q-\la)/4}, \text{ for all } \xi = (x,y,t) \in \H. \text{ \QED}
	\end{align}
\end{Proposition}
Now we recall the following result from \cite{leibfrank} on the square root of convolution of $|\xi|^{-\la}$.
\begin{Lemma}\label{lemhg1}
	Let $0<\la <Q$. There is a function $k \in L_{\emph{w}} ^{\frac{2Q}{Q+\la}}(\H)$ such that 
	\begin{align*}
	|\eta^{-1}\xi |^{-\la}= \int_{\H} k(\eta^{-1}\sigma)k(\xi^{-1}\sigma )~d \sigma, \text{ for all } \xi, \eta \in \H,
	\end{align*}
	where the function $k$ is a real valued, even and homogeneous of degree $-(Q+\la)/2$. 
\end{Lemma}
\begin{Lemma}
		Assume   $0<\la<Q$ and $X_{FL}=\{  u : \H \ra \mathbb{R} \; : \; \|u\|_{FL}<  \infty \}$, where
	\begin{align*}
	\|u\|_{FL}:= \left(\ds \int_{\H}\int_{\H}\frac{|u|^{Q^{*}_{\la}}|u|^{Q^{*}_{\la}}}{|\eta^{-1}\xi|^{\la}}~d\eta d\xi\right)^{\frac{1}{2.Q^{*}_{\la}}}.
	\end{align*}
Then $\|\cdot\|_{FL}$ defines a norm on  $X_{FL}$. Furthermore, $(X_{FL}, \|\cdot \|_{FL})$ is a Banach space. 
\end{Lemma}
\begin{proof}
	With the help of Lemma \ref{lemhg1}, we have 
	\begin{align*}
	\int_{\H}\int_{\H}\frac{|u(\xi)|^{Q^{*}_{\la}}|u(\eta)|^{Q^{*}_{\la}}}{|\eta^{-1}\xi|^{\la}}~d\eta d\xi & = \int_{\H}(|\xi|^{-\la}* |u|^{Q^{*}_{\la}})|u|^{Q^{*}_{\la}}~d\xi \\
	& = \int_{\H}\left( k* |u|^{Q^*_\la}  \right)^2 ~ d\xi
	\end{align*}
	Now  by using Minkowski's inequality, we obtain,  for all $\xi \in \H$, 
	\begin{align*}
	\left( k* |u+v|^{Q^*_\la}  \right)^2 &= \left( \int_{\H} k(\eta^{-1}\xi) |u(\eta)+v(\eta)|^{Q^*_\la}~d \eta  \right)^{\frac{1}{Q^*_\la}\cdot 2 Q^*_\la}\\
	& \leq \left(\left( \int_{\H} k(\eta^{-1}\xi) |u(\eta)|^{Q^*_\la}~d \eta  \right)^{\frac{1}{Q^*_\la}}+ \left( \int_{\H} k(\eta^{-1}\xi) |v(\eta)|^{Q^*_\la}~d \eta  \right)^{\frac{1}{Q^*_\la}}\right)^{2 Q^*_\la}\\
	& = \left(\left( \int_{\H} k(\eta^{-1}\xi) |u(\eta)|^{Q^*_\la}~d \eta  \right)^{\frac{2}{2Q^*_\la}}+ \left( \int_{\H} k(\eta^{-1}\xi) |v(\eta)|^{Q^*_\la}~d \eta  \right)^{\frac{2}{2Q^*_\la}}\right)^{2 Q^*_\la}. 
	\end{align*}
	Once again using Minkowski's inequality, we have 
	\begin{align*}
	& \left(\int_{\H}  \left(  k* |u+v|^{Q^*_\la}  \right)^2 ~ d\xi \right)^{\frac{1}{2Q^*_\la}}
\\ &	\leq \left( \left( \int_{\H} k(\eta^{-1}\xi) |u(\eta)|^{Q^*_\la}~d \eta  \right)^{2}	\right)^{\frac{1}{2Q^*_\la}} + \left( \left( \int_{\H} k(\eta^{-1}\xi) |v(\eta)|^{Q^*_\la}~d \eta  \right)^{2}	\right)^{\frac{1}{2Q^*_\la}}\\
	& =  \left(\int_{\H}  \left(  k* |u|^{Q^*_\la}  \right)^2 ~ d\xi \right)^{\frac{1}{2Q^*_\la}}+ \left(\int_{\H}  \left(  k* |v|^{Q^*_\la}  \right)^2 ~ d\xi \right)^{\frac{1}{2Q^*_\la}}
	\end{align*}
	It implies $\|u+v\|_{FL}\leq  \|u\|_{FL}+\|v\|_{FL}$, for all $u, v \in L^{Q^*}(\H)$. Other properties trivial to check that $\|\cdot\|_{FL}$ is a norm. By standard arguments, one can prove the completeness of the space $X_{FL}.$ This completes the proof of the Lemma. \QED
\end{proof}
With the help of Proposition \ref{prophg1}, the following integral 
\begin{align}\label{hg35}
\int_{\H}\int_{\H}\frac{|u(\xi)|^\al |u(\eta)|^\al}{|\eta^{-1}\xi|^{\la}}~d\eta d\xi 
\end{align}
is well defined if $|u|^\al\in L^r(\H)$ for some $r>1$ such that 
$
\frac{2}{r}+\frac{\la}{Q}=2.$
It implies that if $u \in S_1^2(\H)$ then by using Sobolev-type inequality, \eqref{hg35} is defined only if 
\begin{align*}
2\leq r\al\leq \frac{2Q}{Q-2}.
\end{align*}
That is, 
\begin{align*}
\frac{2Q-\la}{Q}\leq \al\leq \frac{2Q-\la}{Q-2}:= Q^*_\la.
\end{align*}
Therefore from Proposition \ref{prophg1}, we have 
\begin{align*}
\int_{\H}\int_{\H}\frac{|u(\xi)|^{Q^*_\la} |u(\eta)|^{Q^*_\la}}{|\eta^{-1}\xi|^{\la}}~d\eta d\xi \leq C(Q,\la) |u|_{Q^*}^{2.Q^*_\la}. 
\end{align*}
We define 
\begin{align}\label{hg42}
S_{HG}= \inf_{ u \in \mathring{S}_1^2(\H)\setminus\{0\}} \frac{\ds \int_{\H}|\na u |^2~ d\xi}{\left(\ds \int_{\H}\int_{\H}\frac{|u(\xi)|^{Q^*_\la} |u(\eta)|^{Q^*_\la}}{|\eta^{-1}\xi|^{\la}}~d\eta d\xi\right)^{\frac{1}{Q^*_\la}}}
\end{align}
as the best constant. Now we have the following Lemma which will provide the minimizers of $S_{HG}$ and a relation between $S$ and $S_{HG}$. 
\begin{Lemma}\label{lemhg3}
		The best constant $S_{HG}$ is achieved if and only if 
	 \begin{align*}
	 u(\xi)=u(x,y,t)= CZ(\de_\theta(a^{-1}\xi)) ,
	 \end{align*}
	 where $C>0$ is a fixed constant, $ a\in \H,\; \theta \in (0,\infty)$ are parameters and $Z$ is defined in \eqref{hg4}. Moreover,
	 \begin{align*}
	 S_{HG}= S \left(C(Q,\la)\right)^{\frac{-1}{Q^*_\la}},
	 \end{align*}
	where $S$ is  the best  constant defined in \eqref{hg3}. 
\end{Lemma}
\begin{proof}
	Since equality in \eqref{hg1} holds if and only if 
	 	\begin{align*}
f(\xi)=	h(\xi)= cU(\de_\theta(a^{-1}\xi)),
	\end{align*} 
	 where $U$ is defined in \eqref{hg34}, $c>0,\; a\in \H$ and  $\theta \in (0,\infty)$ are parameters. Hence 
	 \begin{align*}
	\left( \int_{\H}\int_{\H}\frac{|u(\xi)|^{Q^*_\la} |u(\eta)|^{Q^*_\la}}{|\eta^{-1}\xi|^{\la}}~d\eta d\xi \right)^{\frac{1}{Q^*_\la}} =  C(Q,\la)^{\frac{1}{Q^*_\la}} |u|_{Q^*}^{2}. 
	 \end{align*}
	 if and only if 
	  \begin{align}\label{hg5}
	 u(\xi)=u(x,y,t)= CZ(\de_\theta(a^{-1}\xi)) .
	 \end{align}
	 Indeed if
	  \[
	W(\xi)= S^{\frac{(Q-\la)(2-Q)}{4(Q-\la+2)}}C(Q,\la)^{\frac{2-Q}{2(Q-\la+2)}}Z(\xi),
	\]
	then $W$ is unique minimizer of $S_{HG}$ and satisfies the following:
	\begin{equation*}
	 \begin{aligned}
	& -\De u = \left(\int_{\H}\frac{|u(\eta)|^{Q^*_\la}}{|\eta^{-1}\xi|^{\la}}d\eta\right)|u|^{Q^*_\la-2}u  \text{ in } \H, \\
	&  \int_{\H} |\na W|^2~ d \xi= \int_{\H}\int_{\H}\frac{|W(\xi)|^{Q^*_\la} |W(\eta)|^{Q^*_\la}}{|\eta^{-1}\xi|^{\la}}~d\eta d\xi = S_{HG}^{\frac{2Q-\la}{Q-\la+2}}.
	\end{aligned}
	\end{equation*}
	 By using the definition of $S_{HG}$ and \eqref{hg5}, we obtain 
	\begin{equation}\label{hg6}
	\begin{aligned}
S_{HG}& \leq   \frac{\ds \int_{\H}|\na u |^2~ d\xi}{\left(\ds \int_{\H}\int_{\H}\frac{|u(\xi)|^{Q^*_\la} |u(\eta)|^{Q^*_\la}}{|\eta^{-1}\xi|^{\la}}~d\eta d\xi\right)^{\frac{1}{Q^*_\la}}}\\
& = \frac{\ds \int_{\H}|\na u |^2~ d\xi}{C(Q,\la)^{\frac{1}{Q^*_\la}} |u|_{Q^*}^{2}}
 = S \left(C(Q,\la)\right)^{\frac{-1}{Q^*_\la}}. 
	 \end{aligned}  
	 \end{equation} 
	By Proposition \ref{prophg1}, we have 
	\begin{align}\label{hg7}
		S_{HG} \geq  \inf_{ u \in \mathring{S}_1^2(\H)\setminus\{0\}} \frac{\ds \int_{\H}|\na u |^2~ d\xi}{\left(C(Q,\la)^{\frac{1}{Q^*_\la}} |u|_{Q^*}^{2}\right)}= \frac{S}{C(Q,\la)^{\frac{1}{Q^*_\la}}}. 
	\end{align}
	From \eqref{hg6} and \eqref{hg7}, we get the desired result. \QED
	\end{proof}

\begin{Lemma}\label{lemhg16}
	Let $0< \la< Q$. If $\{u_n\}$ is a bounded sequence in $L^{\frac{2Q}{Q-2}}(\H)$ such that $u_n \ra u$ a.e in $\H$ as $n \ra \infty$. Then as $ n \ra \infty$, the following holds 
	\begin{align*}
	& \int_{\H} \left(  |\xi|^{-\la}* |u_n|^{Q^*_\la}\right) |u_n|^{Q^*_\la}~d \xi- \int_{\H} \left( |\xi|^{-\la}* |u_n-u|^{Q^*_\la}\right) |u_n-u|^{Q^*_\la}~d \xi\\
	& \qquad  \ra \int_{\H} \left( |\xi|^{-\la}* |u|^{Q^*_\la}\right) |u|^{Q^*_\la}~d \xi
	\end{align*}
\end{Lemma}
\begin{proof}
Using the same technique as in  \cite[Lemma 2.2]{yang}, one can prove the result.\QED
\end{proof}
Corresponding to problem  $(P_a)$, thhe energy functional is defined as   $\J:\mathring{S}_1^2(\Om) \rightarrow \mathbb R$ as
\begin{align*}
\J(u)= \frac{1}{2 }\int_{\Om}|\na u |^2 ~d \xi- \frac{a}{2}\int_{\Om}|u|^2~ d \xi - \frac{1}{2.Q^*_\la} \int_{\Om}\int_{\Om}\frac{|u(\xi)|^{Q^*_\la} |u(\eta)|^{Q^*_\la}}{|\eta^{-1}\xi|^{\la}}~d\eta d\xi
\end{align*}
In consequence of Proposition \ref{prophg1}, we have $\J \in C^1(\mathring{S}_1^2(\Om), \mathbb{R})$. Moreover, for all $ v \in \mathring{S}_1^2(\Om)$, we have 
\begin{align*}
\ld \J^\prime (u), v \rd = \int_{\Om} \na u \cdot \na v~ d\xi - a \int_{\Om} uv ~d \xi - \int_{\Om}\int_{\Om}\frac{|u(\xi)|^{Q^*_\la} |u(\eta)|^{Q^*_\la-2}u(\eta) v(\eta)}{|\eta^{-1}\xi|^{\la}}~d\eta d\xi.
\end{align*} 
\begin{Lemma}\label{lemhg15}
	Let $\la\in(0,  Q)$ and $a>0$. Then every Palais-Smale sequence is bounded. Moreover, if $u_0$ is the weak limit of a Palais-Smale sequence then $u_0$ is a weak solution of $(P_a)$.
\end{Lemma}
\begin{proof}
	Let $\{u_n\}$ be a Palais-Smale sequence of $\J$. That is, there exists a  constant $C>0$ such that  $|\J(u_n)|\leq  C $ and  $\bigg|\bigg \ld \J^\prime (u_n), \ds \frac{u_n}{\|u_n\|} \bigg\rd \bigg|\leq C$. It implies that 
	\begin{align}\label{hg36}
	\|u_n\|_{FL}^{2Q^*_\la}= \int_{\Om}\int_{\Om}\frac{|u_n(\xi)|^{Q^*_\la} |u_n(\eta)|^{Q^*_\la}}{|\eta^{-1}\xi|^{\la}}~d\eta d\xi\leq C(1+ \|u_n\|).
	\end{align}
	First we will  prove that the sequence  $\{ u_n\} $ is bounded. For the sake of  clarity, we consider two cases: \\
	\textbf{Case 1:} If $0<a<a_1$.\\
	With the help of variational characterization of $a_1$, we have 
	\begin{align*}
\left(\frac{1}{2}- \frac{1}{2.Q^*_\la}\right) c \|u_n\|^2  & \leq   \left(\frac{1}{2}- \frac{1}{2.Q^*_\la}\right) (
\|u_n\|^2- a |u_n|^2_2)\\
& = 	\J(u_n)- \frac{1}{2.Q^*_\la}\ld \J^\prime (u_n), \ds u_n \rd\\
&  \quad \leq  C(1+\|u_n\|), 
	\end{align*}
for some $c>0$. It implies $\{ u_n\}$ is a bounded sequence. \\
\textbf{Case 2:} If $a_k\leq a < a_{k+1}$ for some $k \in \mathbb{N}$.\\
We can choose  $\al < \frac{1}{2.Q^*_\la}< \frac{1}{2},\; u_n= v_n+w_n$ where $ v_n \in \mathbb{Y}_r$ and $w_n \in \mathbb{E}_{r+1} $ for some $ r \in \mathbb{N}$. Then 
\begin{align*}
C(1+\|u_n\|)& \geq \J(u_n)- \al\ld \J^\prime (u_n), \ds u_n \rd\\
& =  \left(\frac{1}{2}- \al\right)  (
\|u_n\|^2- a |u_n|^2_2) +  \left(\al- \frac{1}{2.Q^*_\la}\right) \|u_n\|_{FL}^{2.Q^*_\la}\\
& =  \left(\frac{1}{2}- \al\right)  (
\|v_n\|^2+\|w_n\|^2 - a |v_n|^2_2- a |w_n|^2_2) +  \left(\al- \frac{1}{2.Q^*_\la}\right) \|u_n\|_{FL}^{2.Q^*_\la}\\
& \geq \left(\frac{1}{2}- \al\right)  (
\|w_n\|^2- \frac{a}{a_{r+1}} \|w_n\|^2 + (a_1-a) |v_n|^2_2) +  \left(\al- \frac{1}{2.Q^*_\la}\right) \|u_n\|_{FL}^{2.Q^*_\la}. 
\end{align*}
Thus using  \eqref{hg36},
\begin{align*}
C_2\|u_n\|^2 \leq C_3(1+\|u_n\|)+ C_4 (|v_n|^2_2+\|v_n\|^2).
\end{align*}
	Since $Y_r$ is finite dimensional, $|v_n|_2, \; \|v_n\|_2$ and  $\|v_n \|_{FL}$ all are equivalent.  Therefore,
	\begin{align*}
	C_2\|u_n\|^2 \leq C_3(1+\|u_n\|)+ C_5 (1+\|u_n\|)^{\frac{1}{Q^*_\la}},
	\end{align*}
	for some appropriate positive constants $C_2,C_3,C_4$ and $C_5$. It implies that there exists a constant $K>0$ such that $\|u_n \|< K $ for all $n \in \mathbb{N}$. In addition, there exists $u_0\in \mathring{S}_1^2(\Om)$ such that $u_n \rp u_0$ weakly in $\mathring{S}_1^2(\Om)$ and $u_n \ra u_0$ strongly in $L^r(\Om)$ for all $r \in [1, Q^*)$. Next, we prove that $u_0$ is a weak solution of $(P_a)$. 
	Since  $u_n \rp u_0 $   weakly in   $\mathring{S}_1^2(\Om)$ implies $|u_n|^{Q^*_{\la}} \rp |u_0|^{Q^*_{\la}}$ weakly in $L^{\frac{2Q}{2Q-\la}}(\Om)$. By Proposition \ref{prophg1}, Riesz potential on Heisenberg group defines a linear continuous map from   $L^{\frac{2Q}{2Q-\la}}(\Om)$ to $L^{\frac{2Q}{\la}}(\Om)$, we deduce that
	\[
	\int_{\Om}\frac{|u_n(\eta)|^{Q^*_{\la}}}{|\eta^{-1}\xi|^{\la}}d\eta \rp \int_{\Om}\frac{|u_0(\eta)|^{Q^*_{\la}}}{|\eta^{-1}\xi|^{\la}}d\eta  \; \text{ weakly in }\; L^{\frac{2Q}{\la}}(\Om)\]
	Also, $|u_n|^{Q^*_{\la}-2}u_n \rp |u_0|^{Q^*_{\la}-2}u_0 $ weakly in $L^{\frac{2Q}{Q-\la+2}}(\Om)$.
Resuming the information collected so far, we have gained that 
	\[\left(\int_{\Om}\frac{|u_n(\eta)|^{Q^*_{\la}}}{|\eta^{-1}\xi|^{\la}}d\eta\right) |u_n|^{2^{*}_{\la}-2} u_n \rp \left(  \int_{\Om}\frac{|u_0(\eta)|^{Q^*_{\la}}}{|\eta^{-1}\xi|^{\la}}d\eta\right) |u_0|^{2^{*}_{\la}-2} u_0 \; \text{ weakly in }\; L^{\frac{2Q}{Q+2}}(\Om).\]
	Then, for any $v \in \mathring{S}_1^2(\Om)$, 
	\begin{equation}\label{hg8}
	\begin{aligned}
	\int_{\Om}\int_{\Om}& \frac{|u_n(\xi)|^{Q^*_{\la}}|u_n(\eta)|^{Q^*_{\la}-2}u_n(\eta) v(\eta) }{|\eta^{-1}\xi|^{\la}}~d\xi d\eta
	\\
	& \ra \int_{\Om}\int_{\Om}\frac{|u_0(\xi)|^{Q^*_{\la}}|u_0(\eta)|^{Q^*_{\la}-2}u_0(\eta) v(\eta) }{|\eta^{-1}\xi|^{\la}}~d\xi d\eta \; \text{as}\; n\rightarrow \infty.
	\end{aligned}
	\end{equation}
Consider
	\begin{align*}
	\ld \J^{\prime}(u_n)&-\J^{\prime}(u_0), v \rd = \int_{\Om}\na u_n \cdot \na v d\xi  -\int_{\Om} u_n v d\xi - \int_{\Om}\na u_0 \cdot  \na v d\xi+ \int_{\Om} u_0 v d\xi\\
	&-\int_{\Om}\int_{\Om}\left[\frac{|u_n(\xi)|^{Q^*_{\la}}|u_n(\eta)|^{Q^*_{\la}-2}u_n v(\eta) }{|\eta^{-1}\xi|^{\la}} -
	   \frac{|u_0(\xi)|^{Q^*_{\la}}|u_0(\eta)|^{Q^*_{\la}-2}u_0 v(\eta) }{|\eta^{-1}\xi|^{\la}}\right] d\xi d\eta.
	\end{align*}
	Taking into account the fact that $u_n \rp u_0 $   weakly in   $\mathring{S}_1^2(\Om)$, $u_n \ra u_0$ strongly in $L^r(\Om)$ for all $r \in [1, Q^*)$ and \eqref{hg8}, we have 
		\begin{align*}
	\ld \J^{\prime}(u_n)-\J^{\prime}(u_0), v \rd \ra 0 \text{ as } n \ra \infty. 
	\end{align*}
	Hence, $	\ld \J^{\prime}(u_0), v \rd= 0 $ for all $v \in \mathring{S}_1^2(\Om)$. That is, $u_0$ is a weak solution of $(P_a)$. \QED	
\end{proof}

\begin{Lemma}\label{lemhg7}
Let $0<\la< Q,\; a>0$ and $\{ u_n\}$ be a $(PS)_c$ sequence of $\J$ with 
\begin{align*}
c< \frac{Q+2-\la}{2(2Q-\la)} S_{HG}^{\frac{2Q-\la}{Q+2-\la}}. 
\end{align*}
Then $\{ u_n\}$ has a convergent subsequence. 
\end{Lemma}
\begin{proof}
	We will follow the standard variational technique to proof the Lemma.  Let  $u_0$ be the weak limit of the sequence $u_n$ obtained in Lemma \ref{lemhg15}. We set $z_n=u_n-u_0$, then $z_n\rp 0$ weakly in $\mathring{S}^2_1(\Om)$ and $z_n\ra 0$ a.e in $\Om$ as $n \ra \infty$.  By Brezis-Leib Lemma \cite{bre1983}  and Lemma \ref{lemhg16} we have 
	\begin{align}\label{hg37}
&	\int_{\Om} |\na u_n|^2~ d\xi= 	\int_{\Om} |\na z_n|^2~ d\xi+	\int_{\Om} |\na u_0|^2~ d\xi
\end{align}
 and 
 \begin{equation}
 \begin{aligned}\label{hg38}
  	\int_{\Om}\int_{\Om}& \frac{|u_n(\xi)|^{Q^*_{\la}}|u_n(\eta)|^{Q^*_{\la}} }{|\eta^{-1}\xi|^{\la}}~d\xi d\eta \\
  	& = \int_{\Om}\int_{\Om} \frac{|z_n(\xi)|^{Q^*_{\la}}|z_n(\eta)|^{Q^*_{\la}} }{|\eta^{-1}\xi|^{\la}}~d\xi d\eta+ \int_{\Om}\int_{\Om} \frac{|u_0(\xi)|^{Q^*_{\la}}|u_0(\eta)|^{Q^*_{\la}} }{|\eta^{-1}\xi|^{\la}}~d\xi d\eta+ o_n(1).
	\end{aligned}
	\end{equation}
	 With the help of the fact that $\ld \mathcal{J}_a^\prime(u_0),u_0 \rd = 0$, we deduce that 
	 \begin{align}\label{hg39}
	 \mathcal{J}_a(u_0)= \frac{Q+2-\la}{2(2Q-\la)} \int_{\Om}\int_{\Om} \frac{|u_0(\xi)|^{Q^*_{\la}}|u_0(\eta)|^{Q^*_{\la}} }{|\eta^{-1}\xi|^{\la}}~d\xi d\eta. 
	 \end{align}
	 Taking into  account \eqref{hg37}, \eqref{hg38} and \eqref{hg39}, we get
	\begin{equation}\label{hg40}
	\begin{aligned}
c \leftarrow \mathcal{J}_a(u_n)& = \mathcal{J}_a(u_0)+ \frac{1}{2} \int_{\Om} |\na z_n|^2~d\xi - \frac{1}{2Q^*_\la} \int_{\Om}\int_{\Om} \frac{|z_n(\xi)|^{Q^*_{\la}}|z_n(\eta)|^{Q^*_{\la}} }{|\eta^{-1}\xi|^{\la}}~d\xi d\eta+o_n(1)\\
& \geq \frac{1}{2} \int_{\Om} |\na z_n|^2~d\xi - \frac{1}{2Q^*_\la} \int_{\Om}\int_{\Om} \frac{|z_n(\xi)|^{Q^*_{\la}}|z_n(\eta)|^{Q^*_{\la}} }{|\eta^{-1}\xi|^{\la}}~d\xi d\eta+o_n(1)
\end{aligned}
\end{equation}
	and 
	\begin{equation}\label{hg41}
	\begin{aligned}
	o_n(1)= \ld \mathcal{J}_a^\prime(u_n),u_n\rd=   \int_{\Om} |\na z_n|^2~d\xi -  \int_{\Om}\int_{\Om} \frac{|z_n(\xi)|^{Q^*_{\la}}|z_n(\eta)|^{Q^*_{\la}} }{|\eta^{-1}\xi|^{\la}}~d\xi d\eta+o_n(1).
	\end{aligned}
	\end{equation}
	It implies that 
	\begin{align*}
	\int_{\Om} |\na z_n|^2~d\xi =   \int_{\Om}\int_{\Om} \frac{|z_n(\xi)|^{Q^*_{\la}}|z_n(\eta)|^{Q^*_{\la}} }{|\eta^{-1}\xi|^{\la}}~d\xi d\eta \ra m 
	\end{align*}
	where $m$ is a nonnegative constant. From \eqref{hg40}, \eqref{hg41} and \eqref{hg42}, we obtain 
	\begin{align*}
	c \geq \frac{Q+2-\la}{2(2Q-\la)}m  \text{ and } m=0 \text{ or } m \geq S_{HG}^{\frac{2Q-\la}{Q+2-\la}}.
	\end{align*}
	If $m=0$ we are done, else 
	\begin{align*}
	\frac{Q+2-\la}{2(2Q-\la)}S_{HG}^{\frac{2Q-\la}{Q+2-\la}} \leq c, 
	\end{align*}
	which contradicts  the assumption on the range of $c$. Thus, $m=0$, that is, $\|u_n-u_0\| \ra 0$ as $n\ra \infty$. \QED
\end{proof}
\section{ The coercive case (\bm{$0<a<a_1$})}
 In this section, we will give the prove of Theorem \ref{thmhg1} in case of $0<a<a_1$. In order to prove this we fist give some estimates on the critical term involving minimizers. Finally we proved the existence of a nontrivial solution of $(P_a)$ using the Mountain Pass Lemma.
From Lemma \ref{lemhg3}, we know that the  function 
\begin{align*}
W(\xi)= W(x,y,t)=  \frac{S^{\frac{(Q-\la)(2-Q)}{4(Q-\la+2)}}C(Q,\la)^{\frac{2-Q}{2(Q-\la+2)}}B_0}{(t^2+(1+|x|^2+|y|^2)^2)^{(Q-2)/4}}
\end{align*}
 is a minimizer of $S_{HG}$, where $B_0>0$ is defined in \eqref{hg4}  and $ S^{\frac{-(Q-\la)(2-Q)}{4(Q-\la+2)}}C(Q,\la)^{\frac{Q-2}{2(Q-\la+2)}}W$ is minimizer of $S$. 
Moreover,  for $\ba >0$, the function $W_{\ba}$ defined as 
  $$W_{\ba }= \frac{\beta^{\frac{Q-2}{2}} S^{\frac{(Q-\la)(2-Q)}{4(Q-\la+2)}}C(Q,\la)^{\frac{2-Q}{2(Q-\la+2)}}C }{\left(\ba^4t^2+ (1+\ba^2 |x|^2+\ba^2 |y|^2)^2 \right)^{\frac{Q-2}{4}}}$$ satisfies 
  \begin{align*}
& \int_{\H} |\na W_{\ba }|^2~ d \xi= \int_{\H}\int_{\H}\frac{|W_{\ba }(\xi)|^{Q^*_\la} |W_{\ba }(\eta)|^{Q^*_\la}}{|\eta^{-1}\xi|^{\la}}~d\eta d\xi = S_{HG}^{\frac{2Q-\la}{Q-\la+2}}, \\
& \int_{\H} | W_{\ba }|^{Q^*}~ d \xi= S^{\frac{Q}{Q-\la+2}}C(Q,\la)^{\frac{-Q}{Q-\la+2}}. 
\end{align*}
 Without loss of generality, we can assume that $0 \in \Om$ and there exists $R>0$ such that $B(0,4R) \subset \Om \subset B(0,\kappa R)$ for some $\kappa>0$. Let $\phi \in C_c^\infty(\Om)$  such that $0\leq \phi \leq 1 ,\; |\na \phi |$ is bounded and 
\begin{align*}
 \phi(\xi)= \left\{
\begin{array}{ll}
1, & \text{ if } \xi \in B(0,R)  \\
0, &  \text{ if } \xi \in \H \setminus B(0,2R) . \\
\end{array} 
\right. 
\end{align*}
Then we consider the function $
w_\ba = \phi W_{\ba } \in \mathring{S}_1^2(\Om)$. We have the following asymptotic estimates for $w_\ba$.
\begin{Lemma}\label{lemhg5}
	Let $0<\la <Q$ then the following holds:
	\begin{itemize}
		\item [(i)] $\ds \int_{ \Om}|w_\ba|^2~d\xi  \geq C \left\{
		\begin{array}{ll}
		\ba^{-2}+ O(\ba^{-Q+2}), &  \text{ if }  Q>4, \\
		\ba^{-2} \log\; \ba+O(\ba^{-2}), & \text{ if } Q=4.\\
		\end{array} 
		\right.$ 
		\item[(ii)] $\ds \int_{ \Om}  |w_\ba|^{Q^*}d \xi = S^{\frac{Q}{Q-\la+2}}C(Q,\la)^{\frac{-Q}{Q-\la+2}}+ O(\ba ^{-Q})$. 
		\item[(iii) ]$ \ds \int_{\Om}\int_{\Om}\frac{|w_{\ba}(\xi)|^{Q^*_\la} |w_{\ba }(\eta)|^{Q^*_\la}}{|\eta^{-1}\xi|^{\la}}~d\eta d\xi  \leq  S_{HG}^{\frac{2Q-\la}{Q-\la+2}}+ O(\ba^{-Q})$.
		\item[(iv)]  $ \ds \int_{\Om}\int_{\Om}\frac{|w_{\ba}(\xi)|^{Q^*_\la} |w_{\ba }(\eta)|^{Q^*_\la}}{|\eta^{-1}\xi|^{\la}}~d\eta d\xi  \geq  S_{HG}^{\frac{2Q-\la}{Q-\la+2}}- O(\ba^{-\frac{2Q-\la}{2}})$.
		\item[(v)] $\ds \int_{ \Om} |\na w_\ba |^2 d \xi \leq S_{HG}^{\frac{2Q-\la}{Q-\la+2}}+ O(\ba^{-\min\{ {\frac{2Q-\la}{2}}, Q-2\}} )$. 
	\end{itemize}
\end{Lemma}

\begin{proof}
\textit{(i)}  Consider 
\begin{align*}
\int_{ \Om}|w_\ba|^2~d\xi & = C \int_{ B(0,2R)}\frac{\ba ^{Q-2}\phi(\xi)^2}{(\ba ^4t^2+ (1+|\ba x|^2+ |\ba y |^2)^2)^{\frac{Q-2}{2}}}~d\xi \\
& \geq C\int_{ B(0,R)}\frac{\ba ^{Q-2}}{(\ba ^4t^2+ (1+|\ba x|^2+ |\ba y |^2)^2)^{\frac{Q-2}{2}}}~d\xi \\
& =C  \int_{ B(0,\ba R)}\frac{\ba ^{-2}}{(t^2+ (1+|x|^2+ | y |^2)^2)^{\frac{Q-2}{2}}}~d\xi\\
&\geq  C \ba ^{-2} \left(  \int_{ B(0,1)}~d\xi + \int_{B(0,\ba R) \setminus B(0,1)}\frac{1}{|\xi|^{2(Q-2)}}~d\xi    \right)\\
& \geq C \ba ^{-2} \left(  C_1+ \int_1^{\ba R } \rho^{-Q+3}~d\rho    \right).
\end{align*}
Now, if $Q>4$, then 
\begin{align*}
\int_{ \Om}|w_\ba|^2~d\xi &  \geq C \ba ^{-2} \left(  C_1+ \int_1^{\ba R } \rho^{-Q+3}~d\rho    \right)=  C \ba ^{-2}+  O( \ba ^{2-Q}), 
\end{align*}
and if $Q=4$ then 
\begin{align*}
\int_{ \Om}|w_\ba|^2~d\xi &  \geq C \ba ^{-2} \left(  C_1+ \int_1^{\ba R } \rho^{-1}~d\rho    \right)=  C \ba ^{-2}+ C  \ba ^{-2} |\log \ba|  . 
\end{align*}
\textit{(ii)} Clearly,
\begin{align*}
\int_{ \Om}  |w_\ba|^{Q^*}d \xi& = \int_{\H}  |W_\ba|^{Q^*}d \xi - \int_{\H \setminus B(0,R)}  |W_\ba|^{Q^*}d \xi+ \int_{B(0,2R) \setminus B(0,R)}  |\phi W_\ba|^{Q^*}d \xi\\
& = S^{\frac{Q}{Q-\la+2}}C(Q,\la)^{\frac{-Q}{Q-\la+2}}  - \int_{\H \setminus B(0,R)}  |W_\ba|^{Q^*}d \xi+ \int_{B(0,2R) \setminus B(0,R)}  |\phi W_\ba|^{Q^*}d \xi. 
\end{align*}
Using the definition of $W_\ba $, we deduce that 
\begin{align*}
\int_{\H \setminus B(0,R)}  |W_\ba|^{Q^*}d \xi & = C \int_{\H \setminus B(0,R)} \frac{\ba ^Q}{(\ba ^4t^2+ (1+|\ba x|^2+ |\ba y |^2)^2)^{\frac{Q}{2}}}    d \xi\\
& \leq C \int_{\H \setminus B(0,R)} \frac{\ba ^{-Q}}{(t^2+ (|x|^2+ | y |^2)^2)^{\frac{Q}{2}}}    d \xi\\
& = C \int_{\H \setminus B(0,R)} \frac{\ba ^{-Q}}{|\xi |^{2Q}}    d \xi = O(\ba^{-Q}). 
\end{align*}
Also,
\begin{align*}
\int_{B(0,2R) \setminus B(0,R)}  |\phi W_\ba|^{Q^*}d \xi \leq C\int_{B(0,2R) \setminus B(0,R)}  \ba ^{-Q}~ d \xi \leq O(\ba ^{-Q}).
\end{align*}
\textit{(iii)} Taking into 
account  Proposition \ref{prophg1}, Lemma \ref{lemhg3} and part (ii), we have 
\begin{align*}
\int_{\Om}\int_{\Om}\frac{|w_{\ba}(\xi)|^{Q^*_\la} |w_{\ba }(\eta)|^{Q^*_\la}}{|\eta^{-1}\xi|^{\la}}~d\eta d\xi & \leq C(Q,\la) \left( S^{\frac{Q}{Q-\la+2}}C(Q,\la)^{\frac{-Q}{Q-\la+2}}+ O(\ba ^{-Q})\right)^{\frac{2Q-\la}{Q}}\\& 
 \leq  S_{HG}^{\frac{2Q-\la}{Q-\la+2}}+ O(\ba^{-Q}). 
\end{align*}
\textit{(iv)} 
Since 
\begin{equation}\label{hg14}
\begin{aligned}
& \int_{\Om}\int_{\Om}\frac{|w_{\ba}(\xi)|^{Q^*_\la} |w_{\ba }(\eta)|^{Q^*_\la}}{|\eta^{-1}\xi|^{\la}}~d\eta d\xi \\
& \geq 
\int_{B(0,R)}\int_{B(0,R)}\frac{|W_{\ba}(\xi)|^{Q^*_\la} W_{\ba }(\eta)|^{Q^*_\la}}{|\eta^{-1}\xi|^{\la}}~d\eta d\xi \\
 & = \int_{\H}\int_{\H}\frac{|W_{\ba}(\xi)|^{Q^*_\la} W_{\ba }(\eta)|^{Q^*_\la}}{|\eta^{-1}\xi|^{\la}}~d\eta d\xi -2 \int_{\H \setminus B(0,R)}\int_{B(0,R)}\frac{|W_{\ba}(\xi)|^{Q^*_\la} W_{\ba }(\eta)|^{Q^*_\la}}{|\eta^{-1}\xi|^{\la}}~d\eta d\xi\\
 & \quad  - \int_{\H \setminus B(0,R)}\int_{\H \setminus B(0,R)}\frac{|W_{\ba}(\xi)|^{Q^*_\la} W_{\ba }(\eta)|^{Q^*_\la}}{|\eta^{-1}\xi|^{\la}}~d\eta d\xi \\
 & = S_{HG}^{\frac{2Q-\la}{Q-\la+2}}  - 2 \int_{\H \setminus B(0,R)}\int_{B(0,R)}\frac{|W_{\ba}(\xi)|^{Q^*_\la} W_{\ba }(\eta)|^{Q^*_\la}}{|\eta^{-1}\xi|^{\la}}~d\eta d\xi \\ & \quad -  \int_{\H \setminus B(0,R)}\int_{\H \setminus B(0,R)}\frac{|W_{\ba}(\xi)|^{Q^*_\la} W_{\ba }(\eta)|^{Q^*_\la}}{|\eta^{-1}\xi|^{\la}}~d\eta d\xi 
\end{aligned}
\end{equation}
  With straightforward computations, we have 
 \begin{equation}\label{hg12}
 \begin{aligned}
  & \int_{\H \setminus B(0,R)}\int_{B(0,R)}\frac{|W_{\ba}(\xi)|^{Q^*_\la} W_{\ba }(\eta)|^{Q^*_\la}}{|\eta^{-1}\xi|^{\la}}~d\eta d\xi \\
  & \leq C(Q,\la) \left(\int_{\H \setminus B(0,R)}|W_{\ba}(\xi)|^{Q^*} d\xi\right)^{\frac{2Q-\la}{2Q}} \left(\int_{B(0,R)}|W_{\ba}(\eta)|^{Q^*}d\eta\right)^{\frac{2Q-\la}{2Q}}\\
  & \leq C \left(\int_{\H \setminus B(0,R)}\frac{\ba ^{Q}}{|\ba \xi|^{2Q}} ~ d\xi\right)^{\frac{2Q-\la}{2Q}}\left(\int_{B(0,\ba R)} \frac{d\eta}{|\eta|^{2Q}}\right)^{\frac{2Q-\la}{2Q}}\\
  & \leq C \ba^{\frac{-(2Q-\la)}{2}} \left( \int_R^{\infty} \frac{1}{\rho^{Q+1}}~ d\rho\right)
  \left(\int_{\H} \frac{d\eta}{|\eta|^{2Q}}\right)^{\frac{2Q-\la}{2Q}}\\  & \leq O(\ba^{\frac{-(2Q-\la)}{2}}), 
 \end{aligned}
 \end{equation}
\begin{equation}\label{hg13}
\begin{aligned}
\int_{\H \setminus B(0,R)}\int_{\H \setminus B(0,R)}\frac{|W_{\ba}(\xi)|^{Q^*_\la} W_{\ba }(\eta)|^{Q^*_\la}}{|\eta^{-1}\xi|^{\la}}~d\eta d\xi
&  \leq C(Q,\la) \left(\int_{\H \setminus B(0,R)}|W_{\ba}(\xi)|^{Q^*}d\xi\right)^{\frac{2Q-\la}{Q}}\\
& \leq C \left(  O(\ba ^{-Q})\right)^{\frac{2Q-\la}{Q}}= O(\ba ^{-(2Q-\la)})
\end{aligned}
\end{equation}
From \eqref{hg14}, \eqref{hg12} and \eqref{hg13}, we have the desired result. \\
\textit{(v)} Since 
\begin{align*}
\int_{ \Om} |\na w_\ba |^2 d \xi & \leq \int_{ \H} |\na w_\ba |^2 d \xi\\
&=  \int_{ \H} |\na W_\ba |^2 \phi^2 d \xi+ \int_{ \H} |\na \phi  |^2 W_\ba^2d \xi +2 \int_{ \H} \na \phi \cdot \na W_\ba  W_\ba  \phi d \xi\\
& = \int_{ \H} \na W_\ba \cdot \na (W_\ba \phi^2) ~d \xi+  \int_{ \H} |\na \phi  |^2 W_\ba^2d \xi\\
& =  \int_{\H}\int_{\H}\frac{|W_{\ba}(\xi)|^{Q^*_\la} W_{\ba }(\eta)|^{Q^*_\la} \phi^2(\eta)}{|\eta^{-1}\xi|^{\la}}~d\eta d\xi +  \int_{ \H} |\na \phi  |^2 W_\ba^2d \xi
\end{align*}
\begin{align*}
& =\int_{\H}\int_{\H}\frac{|W_{\ba}(\xi)|^{Q^*_\la} W_{\ba }(\eta)|^{Q^*_\la} }{|\eta^{-1}\xi|^{\la}}~d\eta d\xi \\
&  \quad - \int_{\H}\int_{\H}\frac{|W_{\ba}(\xi)|^{Q^*_\la} W_{\ba }(\eta)|^{Q^*_\la} (1- \phi^2(\eta))}{|\eta^{-1}\xi|^{\la}}~d\eta d\xi +  \int_{ \H} |\na \phi  |^2 W_\ba^2d \xi\\
& \leq S_{HG}^{\frac{2Q-\la}{Q-\la+2}} + \int_{\H}\int_{\H}\frac{|W_{\ba}(\xi)|^{Q^*_\la} W_{\ba }(\eta)|^{Q^*_\la} (1- \phi^2(\eta))}{|\eta^{-1}\xi|^{\la}}~d\eta d\xi \\
& \qquad +  \int_{ \H} |\na \phi  |^2 W_\ba^2d \xi. 
\end{align*} 
A direct computation gives us   
\begin{equation}\label{hg15}
\begin{aligned}
\int_{\H}\int_{\H}\frac{|W_{\ba}(\xi)|^{Q^*_\la} W_{\ba }(\eta)|^{Q^*_\la} (1- \phi^2(\eta))}{|\eta^{-1}\xi|^{\la}}~d\eta d\xi & 
 \leq C \int_{\H}\int_{\H\setminus B(0, R)}\frac{|W_{\ba}(\xi)|^{Q^*_\la} W_{\ba }(\eta)|^{Q^*_\la} }{|\eta^{-1}\xi|^{\la}}~d\eta d\xi\\
&  \leq C \left(\int_{\H \setminus B(0,R)}|W_{\ba}(\xi)|^{Q^*}d\xi\right)^{\frac{2Q-\la}{2Q}}\\
&  \leq  O(\ba ^{\frac{-(2Q-\la)}{2}}),
\end{aligned}
\end{equation}
and 
\begin{equation}\label{hg16}
\begin{aligned}
\int_{ \H} |\na \phi  |^2 W_\ba^2d \xi \le  C \int_{ B(0,2R) \setminus B(0,R)}  W_\ba^2d \xi \leq O(\ba^{-(Q-2)}).
\end{aligned}
\end{equation}
From \eqref{hg15} and \eqref{hg16}, we have the required result. \QED
\end{proof}
\begin{Lemma}\label{lemhg8}
	Let $ 0<\la <Q$ and $a>0$ then there exists $z \in \mathring{S}_1^2(\Om)\setminus \{0\}$ such that 
	\begin{align*}
	\frac{\|z\|^2 -a |z|^2_2}{\|z\|_{FL}^2}< S_{HG}
	\end{align*}
	\end{Lemma}
\begin{proof}
For the sake of clarity, We divide the proof in two cases:\\
	\textbf{Case 1:} $Q>4$.\\
	Using Lemma \ref{lemhg5} for $\ba $ large enough,  we deduce that  
	\begin{align*}
	\frac{\|w_\ba \|^2 -a |w_\ba|^2_2}{\|w_\ba\|_{FL}^2} &  \leq  \frac{S_{HG}^{\frac{2Q-\la}{Q-\la+2}}+ O(\ba^{-\min\{ {\frac{2Q-\la}{2}}, Q-2\}} ) -a\ba^{-2} - O(\ba^{-Q+2})}{\left(S_{HG}^{\frac{2Q-\la}{Q-\la+2}}- O(\ba^{-\frac{2Q-\la}{2}}) \right)^{\frac{Q-2}{2Q-\la}}}\\
	& \leq S_{HG} \left( \frac{1 + O(\ba^{-\min\{ {\frac{2Q-\la}{2}}, Q-2\}} ) -a\ba^{-2} }{\left(1- O(\ba^{-\frac{2Q-\la}{2}}) \right)^{\frac{Q-2}{2Q-\la}}} \right) \\
	& \leq   S_{HG} \left( 1 + C_1 \ba^{-\min\{ {\frac{2Q-\la}{2}}, Q-2\}}  -C_2 \ba^{-2}  \right) < S_{HG}. 
	\end{align*}
	\textbf{Case 2:} $Q=4$.\\
	Once again employing  Lemma \ref{lemhg5}  as before, for $\ba $ large, we deduce
	\begin{align*}
	\frac{\|w_\ba \|^2 -a |w_\ba|^2_2}{\|w_\ba\|_{FL}^2} &  \leq  \frac{S_{HG}^{\frac{8-\la}{4-\la+2}}+ O(\ba^{-\min\{ {\frac{8-\la}{2}}, 2\}} ) -a\ba^{-2} \log\; \ba }{\left(S_{HG}^{\frac{8-\la}{4-\la+2}}- O(\ba^{-\frac{8-\la}{2}}) \right)^{\frac{2}{8-\la}}}\\
	& \leq S_{HG} \left( \frac{1 + C_1 \ba ^{-2}  -C_2\ba^{-2} \log\; \ba }{\left(1- O(\ba^{-\frac{8-\la}{2}}) \right)^{\frac{4-2}{8-\la}}} \right) \\
	& \leq   S_{HG} \left( 1 + C_1 \ba ^{-2}  -C_2\ba^{-2} \log\; \ba \right) < S_{HG}. 
	\end{align*}
	Therefore, we can choose $z:= w_\ba$ with $\ba $ large enough and the conclusion follows. 	\QED
\end{proof}
\begin{Lemma}\label{lemhg6}
Let  	$ 0< \la < Q$ and $ a\in (0, a_1)$. Then $\J$ satisfies the following conditions:
	\begin{enumerate}
		\item [(i)] There exists $\al,\rho>0$ such that $\J(u)\geq \al$ for $\|u\|= \rho$.	
		\item [(ii)] There exists  $e \in \mathring{S}_1^2(\Om)$ with $\|e\|>\rho$ such that $\J(e)<0$.
	\end{enumerate}
\end{Lemma}
\begin{proof}
\textit{(i)} Using the variational characterization of $a_1$ and definition of $S_{HG}$, we have 
	\begin{align*}
	\J(u)\geq \frac{1}{2}\left(1-\frac{a}{a_1}\right)\|u\|^2-\frac{S_{H,L}^{-Q^*_{\la}}}{2.Q^*_{\la}} \|u\|^{2.Q^*_{\la}}
	\end{align*}
	Using the fact that  $2<2.Q^*_{\la}$  and $a\in (0, a_1)$, we can choose $\al,\rho>0$ such that $\J(u)\geq \al$ whenever $\|u\|= \rho$. \\
	\textit{(ii)} Let $u \in \mathring{S}_1^2(\Om)$ then 
	\begin{align*}
	\J(tu)& = \frac{ t^2}{2}\left( \|u\|^2  - a \int_{\Om}|u|^2 ~d\xi \right) - \frac{t^{2.Q^*_{\la}}}{2.Q^*_{\la}} \int_{\Om}\int_{\Om} \frac{|u(x)|^{2^{*}_{\mu}}|u(y)|^{2^{*}_{\mu}}}{|\eta^{-1}\xi|^{\la}}~d\xi d\eta  \ra -\infty \text{ as } t \ra \infty.
	\end{align*}
	Hence we can choose $t_0>0$ (large enough) such that $e:= t_0u$ such that (ii) follows. \QED
\end{proof}
\begin{Proposition}\label{prophg2}
	Using Lemma \ref{lemhg6} and the mountain pass theorem without $(PS)$ condition (See \cite{willem}) there exits a $(PS)$ sequence at the minimax level 
	\begin{align*}
	c^*= \inf_{\ga \in \Gamma }\max_{ t \in [0,1]} \J(\ga (t))>0 
	\end{align*}
	where 
$	\Gamma:= \{ \ga \in C([0,1]), \mathring{S}_1^2(\Om)\; : \;   \ga (0)= 0 , \J(\ga(1))<0 \}.$
	\QED
\end{Proposition}
\textbf{Proof of main theorem \ref{thmhg1}(i):}
 Let $a \in (0, a_1)$ then from Lemma \ref{lemhg8}, there exists  $z \in \mathring{S}_1^2(\Om)\setminus \{0\}$  such that 
\begin{align*}
\frac{\|z\|^2 -a |z|^2_2}{\|z\|_{FL}^2}< S_{HG}
\end{align*}
Hence,
\begin{align*}
0< \max_{ t \in [0,1]} \J(tz)&  = \max_{ t \in [0,1]} \bigg\{  \frac{ t^2}{2}\left( \|u\|^2  - a \int_{\Om}|u|^2 ~d\xi \right) - \frac{t^{2.Q^*_{\la}}}{2.Q^*_{\la}} \int_{\Om}\int_{\Om} \frac{|u(x)|^{2^{*}_{\mu}}|u(y)|^{2^{*}_{\mu}}}{|\eta^{-1}\xi|^{\la}}~d\xi d\eta \bigg\}\\
& = \frac{Q+2-\la}{2(2Q-\la)} \left(\frac{\|z\|^2 -a |z|^2_2}{\|z\|_{FL}^2} \right)^{\frac{2Q-\la}{Q+2-\la}}\\
& < \frac{Q+2-\la}{2(2Q-\la)}S_{HG}. 
\end{align*}
By definition of $c^*$, we have 
\begin{align*}
c^*< \frac{Q+2-\la}{2(2Q-\la)}S_{HG}. 
\end{align*}
Therefore using Proposition \ref{prophg2}, there exists a   $(PS)$ sequence $\{u_n\}$, which on using Lemma \ref{lemhg7} gives a convergent subsequence of $\{u_n\}$ and thus $\J$ has a weak solution of $(P_a)$. \QED

\section{The noncoercive case (\bm{$a> a_1,\; a\neq a_k$} for all \bm{$k \in \mathbb{N}$})}
In this section we will complete the proof of Theorem \ref{thmhg1} by considering the case $a \geq a_1$.  We will make use of Linking Theorem given Rabinowitz \cite{rainbo}. 
\begin{Lemma}\label{lemhg9}
	Let $a \in [a_i , a_{i+1})$ for some $ i \in \mathbb{N}$ then the functional $\J$ satisfies  following: 
	\begin{itemize}
		\item [(i)]  There exists $\al,\rho>0$ such that   if $u \in \mathbb{E}_{i+1}$ with  $\|u\|= \rho$ then  $\J(u)\geq \al$. 
		\item [(ii)] $\J(u) <0$ for all $u \in \mathbb{Y}_i$.
		\item [(iii)] Let $D$ be a finite dimensional space in $\mathring{S}_1^2(\Om)$ then there exists  $ L> \rho $ such that for any $u \in D$ with $\|u \|>L$ we have  $\J(u)<0$.
	\end{itemize}
\end{Lemma}
\begin{proof}
	 \textit{(i)}  Let $ a \in (a_i , a_{i+1})$ then for $u \in \mathbb{E}_{i+1}$. Then by using Proposition \ref{prophg1}, we have 
	 \begin{align*}
	 \J(u) \geq \frac{1}{2}\left(1- \frac{a}{a_{i+1}}\right) \|u\|^2 - \frac{1}{2Q^*_\la}C_1\|u\|^{2Q^*_\la}. 
	 \end{align*}
	 Now using the fact that $2<2Q^*_\la$, we can choose $\al,\; \rho>0$ such that $\J(u) \geq \al$ for $ u \in \mathbb{E}_{i+1}$ and $\|u\|= \rho$. \\
	 \textit{(ii)} Let $u \in \mathbb{Y}_i$ then $ u = \ds \sum_{k=1}^{k=i} b_k e_k $ where $b_k \in \mathbb{R}$ for all $ k= 1,2,\cdots , i$. Then 
	 \begin{align*}
	 \int_{ \Om}|u|^2~ d\xi = \sum_{k=1}^{k=i} b_k^2 \text{ and } \int_{ \Om}|\na u |^2 ~ d\xi = \sum_{k=1}^{k=i} b_k^2 \|e_k\|^2.
	 \end{align*}
	 Combining this with the fact that $ a_k\leq a_i\leq a$, we deduce that 
	 \begin{align*}
	 \J(u) = \frac{1}{2}\sum_{k=1}^{k=i} b_k^2\left(  \|e_k\|^2- a \right) - \frac{1}{2Q^*_\la} \|u\|_{FL}^{2Q^*_\la}< \frac{1}{2}\sum_{k=1}^{k=i} b_k^2\left(  a_k- a \right) \leq 0. 
	 \end{align*}
	 \textit{(iii)} For $ u \in D \setminus \{0\}$,  using the fact that all norms on finite dimensional space  are equivalent  and $a\geq 0$, we get 
	 \begin{align*}
	 \J(u) & \leq \frac{1}{2} \|u\|^2 - \frac{1}{2Q^*_\la} \|u\|_{FL}^{2Q^*_\la} \leq  \frac{1}{2} \|u\|^2 - \frac{C_1}{2Q^*_\la} \|u\|^{2Q^*_\la},
	 \end{align*}
	 for some appropriate positive constant. Hence $\J(u) \ra - \infty$ as $\|u\|\ra \infty$. Thus we can choose $L>\rho$ such that for any $u \in D$ with $\|u\|\geq L,\; \J(u) \leq 0$.  \QED 
\end{proof}
In view of Lemma \ref{lemhg8}, we know that for $\ba$ large enough 
\begin{align*}
\frac{\|w_\ba \|^2 -a |w_\ba|^2_2}{\|w_\ba\|_{FL}^2}< S_{HG}
\end{align*}
Now for $ i \in \mathbb{N}$, we define the linear space 
\begin{align*}
\mathbb{M}_{i,\ba } := \text{ span}\{ e_1,e_2,\cdots,e_i,w_\ba \}
\end{align*}
and set 
\begin{align*}
m_{i,\ba}:= \max_{u \in \mathbb{M}_{i,\ba },\; \|u\|_{FL}=1}\left(\|u\|^2-a|u|_2^2\right)
\end{align*}
\begin{Lemma}
	Let $a \in [a_i,a_{i+1})$ for some $ i \in \mathbb{N}$, then 
	\begin{itemize}
		\item [(i)]	There exists  $u_m\in \mathbb{M}_{i,\ba}$ such that $m_{i,\ba}$ is achieved at $u_m$ and $u_m= z+tw_\ba$, for some  $z \in \mathbb{Y}_i$.
		\item [(ii)] For $\ba$ large enough, the following holds:
		\begin{align*}
		m_{i,\ba}  \leq \left\{
		\begin{array}{ll}
		(a_i-a) |z|^2_2 ,  &  \text{ if }  t=0, \\
		(a_i-a) |z|^2_2 + G_\ba (1+ |w|_\infty O(\ba^{-\frac{Q-2}{2}}))+ O(\ba^{-\frac{Q-2}{2}})|w|_\infty, & \text{ if } t>0\\
		\end{array} ,
		\right.
		\end{align*}
		where $z$ is defined in \textit{(i)} and 
		\begin{align*}
		G_\ba= \frac{\|w_\ba \|^2 -a |w_\ba|^2_2}{\|w_\ba\|_{FL}^2} .
		\end{align*}
	\end{itemize}
\end{Lemma}
\begin{proof}
	Since $\mathbb{M}_{i,\ba}$ is finite dimensional space, so there exists  $0\not \equiv u_m \in \mathbb{M}_{i,\ba}$ such that
	\begin{align*}
	m_{i,\ba} = \|u_m\|^2-a|u_m|_2^2\quad  \text{ and } \quad  \|u_m\|_{FL}=1.
	\end{align*}  
	From definition of $\mathbb{M}_{i,\ba}$ it is clear that $u_m = z+tw_\ba$ where $z \in \mathbb{Y}_i$  and $ t \in \mathbb{R}$. Without loss of generality we can assume that $t \geq 0$ otherwise we can replace $u_m$ with $-u_m$. 
	\textit(ii) If $t=0$ then $u_m= z\in \mathbb{Y}_i$ and 
	\begin{align*}
	m_{i,\ba}= \|u_m\|^2-a|u_m|^2_2\leq (a_i-a) |z|^2_2.
	\end{align*} 
	Consider the case $t>0$. In this case we first claim that   
	\begin{align*}
	\int_{ \Om}\int_{ \Om} \frac{|w_\ba(\xi)|^{Q^*_\la} |w_\ba(\eta)|^{Q^*_{\la}-1}}{|\eta^{-1}\xi|^{\la}}~d \xi d\eta\leq O(\ba^{-\frac{Q-2}{2}}). 
	\end{align*}
	Consider
	\begin{equation}\label{hg17}
	\begin{aligned}
	\int_{ \Om}\int_{ \Om}  \frac{|w_\ba(\xi)|^{Q^*_\la} |w_\ba(\eta)|^{Q^*_{\la}-1}}{|\eta^{-1}\xi|^{\la}}~d \xi d\eta &  \leq \int_{B(0,\kappa R)}\int_{ B(0,\kappa R)} \frac{|W_\ba(\xi)|^{Q^*_\la} |W_\ba(\eta)|^{Q^*_{\la}-1}}{|\eta^{-1}\xi|^{\la}}~d \xi d\eta\\
	& =  \ba^{-\frac{Q-2}{2}}\int_{B(0,\kappa R)}\int_{ B(0,\kappa R)} \frac{~d \xi d\eta }{|\eta^{-1}\xi|^{\la} |\xi|^{2Q-\la}|\eta|^{Q-\la+2}}  
	\\
	& \leq  \ba^{-\frac{Q-2}{2}}\int_{\H}\int_{ \H} \frac{~d \xi d\eta }{|\eta^{-1}\xi|^{\la} |\xi|^{2Q-\la}|\eta|^{Q-\la+2}} . 
	\end{aligned}
	\end{equation}
	Now using Proposition \ref{prophg1} with $\la >1$ we have 
\begin{equation}\label{hg18}
	\begin{aligned}
	\int_{\H}\int_{ \H}  \frac{~d \xi d\eta }{|\eta^{-1}\xi|^{\la} |\xi|^{2Q-\la}|\eta|^{Q-\la+2}}  
	& \leq C\left( \int_{ \H}  \frac{~d \xi  }{ |\xi|^{\frac{(2Q-\la)Q}{(Q-1)}}} \right)^{\frac{Q-1}{Q}}	\left( \int_{\H}\frac{~ d\eta }{ |\eta|^{\frac{(Q-\la+2)Q}{(Q-\la+1)}}} \right)^{\frac{Q-\la+1}{Q}}\\
	&  < + \infty, 
	\end{aligned}
	\end{equation}
	and for $\la\leq 1$ again using Proposition \ref{prophg1} yields 
\begin{equation}\label{hg19}
\begin{aligned}
	\int_{\H}\int_{ \H}  \frac{~d \xi d\eta }{|\eta^{-1}\xi|^{\la} |\xi|^{2Q-\la}|\eta|^{Q-\la+2}}  
	& \leq C\left( \int_{ \H}  \frac{~d \xi  }{ |\xi|^{2Q} }\right)^{\frac{2Q-\la}{2Q}}	\left( \int_{\H}\frac{~ d\eta }{ |\eta|^{\frac{2Q(Q-\la+2)}{2Q-\la}}} \right)^{\frac{2Q-\la}{2Q}}\\
	&  < + \infty.
\end{aligned}
\end{equation}
Inequality \eqref{hg17} coupled with \eqref{hg18} and \eqref{hg19} proves the claim. Taking into account the  fact  that $ e_1,e_2,\cdots, e_i \in L^{\infty}(\Om)$ we have $z \in L^\infty(\Om)$ and since $|x|^p$ is convex for all $p>1$, we obtain the following estimates:
\begin{align*}
1= & \int_{ \Om}\int_{ \Om} \frac{|u_m(\xi)|^{Q^*_\la} |u_m(\eta)|^{Q^*_{\la}}}{|\eta^{-1}\xi|^{\la}}~d \xi d\eta\\
=& \int_{ \Om}\int_{ \Om} \frac{|z+tw_\ba(\xi)|^{Q^*_\la} |z+tw_\ba(\eta)|^{Q^*_{\la}}}{|\eta^{-1}\xi|^{\la}}~d \xi d\eta\\
 \geq& \int_{ \Om}\int_{ \Om} \frac{|tw_\ba(\xi)|^{Q^*_\la} |tw_\ba(\eta)|^{Q^*_{\la}}}{|\eta^{-1}\xi|^{\la}}~d \xi d\eta+ 2Q^*_\la \int_{ \Om}\int_{ \Om} \frac{|tw_\ba(\xi)|^{Q^*_\la} |tw_\ba(\eta)|^{Q^*_{\la}-2}tw_\ba(\eta) z(\eta)}{|\eta^{-1}\xi|^{\la}}~d \xi d\eta\\
 \geq& \int_{ \Om}\int_{ \Om} \frac{|tw_\ba(\xi)|^{Q^*_\la} |tw_\ba(\eta)|^{Q^*_{\la}}}{|\eta^{-1}\xi|^{\la}}~d \xi d\eta- 2Q^*_\la|z|_\infty \int_{ \Om}\int_{ \Om} \frac{|tw_\ba(\xi)|^{Q^*_\la} |tw_\ba(\eta)|^{Q^*_{\la}-2}tw_\ba(\eta) }{|\eta^{-1}\xi|^{\la}}~d \xi d\eta. 
\end{align*}
That is,
\begin{equation*}
\begin{aligned}
\int_{ \Om}\int_{ \Om} & \frac{|tw_\ba(\xi)|^{Q^*_\la} |tw_\ba(\eta)|^{Q^*_{\la}}}{|\eta^{-1}\xi|^{\la}}~d \xi d\eta\\
& \leq 1+ 2Q^*_\la|z|_\infty \int_{ \Om}\int_{ \Om} \frac{|tw_\ba(\xi)|^{Q^*_\la} |tw_\ba(\eta)|^{Q^*_{\la}-2}tw_\ba(\eta) }{|\eta^{-1}\xi|^{\la}}~d \xi d\eta\\
&\quad \leq  1+ t^{2Q^*_\la-1}O(\ba^{-\frac{Q-2}{2}}). 
\end{aligned}
\end{equation*}
	It implies there exists a constant $c_1>0$ such that $t< c_1$. Hence 
	\begin{equation}\label{hg20}
	\begin{aligned}
	\int_{ \Om}\int_{ \Om}  \frac{|tw_\ba(\xi)|^{Q^*_\la} |tw_\ba(\eta)|^{Q^*_{\la}}}{|\eta^{-1}\xi|^{\la}}~d \xi d\eta  \leq  1+ |z|_\infty O(\ba^{-\frac{Q-2}{2}}). 
	\end{aligned}
	\end{equation}
	 As a consequence of \eqref{hg20} and  definition of $a_i$, we get 
\begin{equation*}
\begin{aligned}
	m_{i, \ba} &=  \|z+tw_\ba\|^2-a|z+tw_\ba|^2_2\\
& 	\leq (a_i-a) |z|^2_2 +G_\ba \left(  \int_{ \Om}\int_{ \Om}  \frac{|tw_\ba(\xi)|^{Q^*_\la} |tw_\ba(\eta)|^{Q^*_{\la}}}{|\eta^{-1}\xi|^{\la}}~d \xi d\eta \right)^{\frac{Q-2}{2Q-\la}}+ C|z|_\infty |w_\ba|_1\\
& \leq (a_i-a) |z|^2_2 +G_\ba \left(  1+ |z|_\infty O(\ba^{-\frac{Q-2}{2}}) \right)+ C|z|_\infty |w_\ba|_1,
		\end{aligned}
	\end{equation*}
for some suitable constant $C>0$. By some straightforward calculations, we can show that $|w_\ba |_1\leq O(\ba^{-\frac{Q-2}{2}})$. Therefore, 
	\begin{equation*}
	m_{i, \ba}  \leq (a_i-a) |z|^2_2 +G_\ba \left(  1+ |z|_\infty O(\ba^{-\frac{Q-2}{2}}) \right)+ C|z|_\infty O(\ba^{-\frac{Q-2}{2}}). \qquad \qquad \text{ \QED}
	\end{equation*}	
\end{proof}
\begin{Lemma}\label{lemhg10}
	Let $ a \in (a_i, a_{i+1}) $ for some $i \in \mathbb{N}$ then for any $u \in \mathbb{M}_{i,\ba}$, 
		\begin{align*}
	\frac{\|u\|^2 -a |u|^2_2}{\|u\|_{FL}^2}< S_{HG}.
	\end{align*}
\end{Lemma}
\begin{proof}
	It is enough to show that 
	\begin{align*}
	m_{i,\ba}= \max_{u \in \mathbb{M}_{i,\ba },\; \|u\|_{FL}=1}\left(\|u\|^2-a|u|_2^2\right) <S_{HG}
	\end{align*}
	If $t=0$ then by using the fact that $ a \in (a_i, a_{i+1}) $ and  \eqref{hg20}, we get that 
\begin{align*}
m_{i,\ba}  \leq (a_i-a) |z|^2_2 < S_{HG}. 
\end{align*}
If $t>0$ we divide the proof in two cases viz, $Q>4$ and $Q=4$. If $Q>4$ then  by using the proof of case 1 of Lemma \ref{lemhg8}, 
 \begin{equation*}
 \begin{aligned}
 m_{i, \ba} & 
  \leq (a_i-a) |z|^2_2 +\frac{\|w_\ba \|^2 -a |w_\ba|^2_2}{\|w_\ba\|_{FL}^2} \left(  1+ |z|_\infty O(\ba^{-\frac{Q-2}{2}}) \right)+ C|z|_\infty O(\ba^{-\frac{Q-2}{2}})\\
 &  \leq (a_i-a) |z|^2_2 +   S_{HG} \left( 1 + C_1 \ba^{-\min\{ {\frac{2Q-\la}{2}}, Q-2\}}  -C_2 \ba^{-2}  \right)  \left(  1+ |z|_\infty O(\ba^{-\frac{Q-2}{2}}) \right)\\ &\quad + C|z|_\infty O(\ba^{-\frac{Q-2}{2}})\\
 &  \leq S_{HG} +   (a_i-a) |z|^2_2 +      C_1 \ba^{-\min\{ {\frac{2Q-\la}{2}},  \frac{Q-2}{2}  \}}\\
 & < S_{HG},  
 \end{aligned}
 \end{equation*}
 for $ \ba $  sufficiently large enough. If $Q=4$ then again utilizing the proof of Lemma \ref{lemhg8}, we have 
  \begin{equation*}
 \begin{aligned}
 m_{i, \ba} & 
 \leq (a_i-a) |z|^2_2 +\frac{\|w_\ba \|^2 -a |w_\ba|^2_2}{\|w_\ba\|_{FL}^2} \left(  1+ |z|_\infty O(\ba^{-1}) \right)+ C|z|_\infty O(\ba^{-1})\\
 &  \leq (a_i-a) |z|^2_2 +    S_{HG} \left( 1 + C_1 \ba ^{-2}  -C_2\ba^{-2} \log\; \ba \right)  \left(  1+ |z|_\infty O(\ba^{-1}) \right)+ C|z|_\infty O(\ba^{-1})\\
 &  \leq S_{HG} +   (a_i-a) |z|^2_2 +      C|z|_\infty O(\ba^{-1})\\
 & < S_{HG},  
 \end{aligned}
 \end{equation*}
 for $\ba$ large. \QED
\end{proof}
\textbf{Proof of main Theorem \ref{thmhg1}(i):}   Let $a > a_1$. At this stage, we write $u_m= \hat{z}+t v_\ba$, where 
\begin{align*}
& \hat{z}= z+ t \sum_{k=1}^{k=i} \left( \int_{ \Om} w_\ba e_k~d \xi \right)e_k \in \mathbb{Y}_i\\
\text{ and } \quad & v_\ba = w_\ba - t \sum_{k=1}^{k=i} \left( \int_{ \Om} w_\ba e_k~d \xi \right)e_k. 
\end{align*}
This implies $\hat{z}$ and $v_\ba$ are orthogonal in $L^2(\Om)$. Moreover,
\begin{align*}
|u_m|^2_2= |\hat{z}|^2_2+ t^2 |v_\ba|^2_2.
\end{align*}
Therefore, we can write 
\begin{align*}
\mathbb{M}_{i,\ba}= \mathbb{Y}_i \oplus \mathbb{R} v_\ba. 
\end{align*}
From Lemma \ref{lemhg9}, we have the following:
\begin{enumerate}
	\item [(i)] $\ds \inf_{u \in \mathbb{E}_{i+1}, \|u\|= \rho} \J(u)\geq  \al >0$, where $\al$ and $\rho$ is defined in Lemma \ref{lemhg9}. 
	\item [(ii)] $\ds \sup_{u \in \mathbb{Y}_i} \J(u)<0$  and  $\ds \sup_{u \in \mathbb{M}_{i,\ba}, \|u\|>\rho} \J(u) \leq 0$. 
\end{enumerate}
Therefore by Linking Theorem \cite[Theorem 5.3]{rainbo}, $\J$ possesses a critical point $c^* \geq \al$ which can be characterized by 
\begin{align*}
c^*= \inf_{\ga \in \Gamma}\sup_{u \in V} \J(\ga(u)),
\end{align*}  
where 
\begin{align*}
V:= (\overline{B}_\rho\cap \mathbb{Y}_i) \oplus \{ rv_\ba\; :\; r \in (0, \rho)  \}
\text{ and } \Gamma: = \{  \ga \in C(\overline{V},\mathring{S}_1^2(\Om))\; :\; \ga = \text{ id on  } \pa V \}. 
\end{align*}
By the definition of  $c^*$, for any $ \ga \in \Gamma$ we have 
$c^*\leq \ds \max_{u \in V} \J(\ga(u))$. In particular, if we take $ \ga= \text{ id}$ on $\overline{V}$ then $c^*\leq \ds \max_{u \in V} \J(u) \leq \max_{u \in \mathbb{M}_{i,\ba}} \J(u) $. By using the fact that $\mathbb{M}_{i,\ba}$ is a linear space, we deduce that
\begin{align*}
\max_{u \in \mathbb{M}_{i,\ba} }\J(u)= \max_{u \in \mathbb{M}_{i,\ba}, t \not = 0} \J\left(\frac{|t|u}{|t|}\right)=   \max_{u \in \mathbb{M}_{i,\ba},t>0} \J(tu) \leq \max_{u \in \mathbb{M}_{i,\ba}, t \geq 0} \J(tu).
\end{align*}
Now observe that for any $ u \in \mathring{S}_1^2(\Om)$,
\begin{align*}
\max_{t \geq 0} \J(tu) = \frac{Q+2-\la}{2(2Q-\la)} \left(  \frac{\|u\|^2 -a |u|^2_2}{\|u\|_{FL}^2} \right)^{\frac{2Q-\la}{Q+2-\la}},
\end{align*}
combining this with Lemma \ref{lemhg10}, we have 
\begin{align*}
c^* & \leq \max_{u \in \mathbb{M}_{i,\ba}, t \geq 0} \J(tu) = \max_{u \in \mathbb{M}_{i,\ba}} \frac{Q+2-\la}{2(2Q-\la)} \left(  \frac{\|u\|^2 -a |u|^2_2}{\|u\|_{FL}^2} \right)^{\frac{2Q-\la}{Q+2-\la}} < \frac{Q+2-\la}{2(2Q-\la)} S_{HG}^{\frac{2Q-\la}{Q+2-\la}}.
\end{align*}
Therefore, by employing Lemma \ref{lemhg7} coupled with Linking theorem, we conclude that there exists a nontrivial solution  of $(P_a)$ in $\mathring{S}_1^2(\Om)$. \QED

\section{Regularity and Non-existence}
In this section we will give the regularity of solutions and nonexistence of solutions. Throughout this section we shall  always denote by $u$  a  weak solution of $(P_a)$ and  $\Om$   to be smooth bounded domain in $\H$ with $C^2$ boundary.  In subsection 5.1, our goal is to prove $L^p$, H\"older continuity properties of  weak solution $u$ to  $(P_a)$. 
We will seek help of  iteration techniques and boot-strap method  given by  Brezis-Kato (See \cite{kato})  and Moser (See \cite[Chapter 8]{gilbarg}). In case of Choquard equation in Euclidean domain, Brezis-Kato type of result was given by Moroz and schaftingen (\cite{moroz2}). In subsection 5.2 we gave a Pohozaev type identity (See Lemma \ref{lemhg14}) for convolution type nonlinearity for Kohn Laplacian which itself a novel  result and of autonomous interest. We also give the proof of Theorem \ref{thmhg2}. 
\subsection{Regularity Results}
\begin{Lemma}\label{lemhg11}
	Let $0<\la<Q$ and $\theta \in (0,2)$. If $F,\; G \in L^{\frac{2Q}{Q-\la+2}}(\H) + L^{\frac{2Q}{Q-\la}}(\H) $ and $ \frac{Q-\la}{Q} < \theta <2- \frac{Q-\la}{Q}$ then for every $\e>0$ there exists $C_{\e,\theta} \in \mathbb{R}$ such that for every $u \in S_1^2(\H)$,
	\begin{align*}
	\int_{ \H}(|\xi|^{-\la}* (F|u|^{\theta}))G|u|^{2-\theta}~ d\xi \leq \e^2\int_{ \H}|\na u|^2 ~ d\xi + C_{\e,\theta} \int_{ \H}|u|^2~ d\xi.
	\end{align*}
\end{Lemma}
\begin{proof}
	Using the same arguments as in \cite[Lemma 3.2]{moroz2} one can easily proof the lemma. All one need to do is to replace $\R$ with  $\H$. \QED
\end{proof}

\begin{Proposition}\label{prophg3}
	If  $ F, G \in L^{\frac{2Q}{Q-\la}}(\H)+ L^{\frac{2Q}{Q-\la+2}}(\H)$ and $u \in S_1^2(\H)$ is a solution of 
	\begin{align}\label{hg22}
	-\De u +u =  (|\xi|^{-\la}* Fu)G,
	\end{align}
	then
	\begin{align*}
	u \in L^s(\H) \text{ for every } s \in \bigg[2, \frac{2Q^2}{(Q-2)(Q-\la)}\bigg) .
	\end{align*}
\end{Proposition}
\begin{proof}
Using Lemma \ref{lemhg11} for $\theta=1$, there exists $\mu >0$ such that for any $ \phi \in S_1^2(\H)$,
	\begin{align*}
\int_{ \H}(|\xi|^{-\la}* (F|\phi|))G|\phi|~ d\xi  \leq \frac{1}{2}\int_{ \H}|\na \phi |^2 ~ d\xi + \frac{\mu}{2} \int_{ \H}|\phi|^2~ d\xi. 
	\end{align*}
Pick sequences $\{F_n\}_{n \in \mathbb{N}}, \{G_n\}_{n \in \mathbb{N}}$  in $L^{\frac{2Q}{Q-\la}}(\mathbb{H}^N)$ such that $|F_n|\leq |F|,\; |G_n| \leq |G|$and $F_n \ra F,\; G_n \ra G$ almost everywhere in $\H$. For each $n \in \mathbb{N},\; \phi \in S_1^2(\H),\;  \psi \in S_1^2(\H)$, define 
\begin{align*}
\Im_n(\phi, \psi)= \int_{ \H}\left( \na \phi \cdot \na \psi  + \mu \phi \psi \right) ~ d \xi- \int_{ \H} \int_{ \H}(|\xi|^{-\la}* F\phi)G\psi~ d\xi.  
\end{align*}
Then $\Im_n: S_1^2(\H) \times S_1^2(\H) \ra \mathbb{R}$ is a bilinear and  coercive map. Hence, by Lax-Miligram theorem, there exists a unique solution of $u_n \in S_1^2(\H)$ of 
\begin{align}\label{hg23}
-\De u_n +\mu u_n =  (|\xi|^{-\la}* F_nu_n)G_n+ (\mu-1)u,
\end{align}
where $u \in S_1^2(\H)$ is the given solution of \eqref{hg22}. By the uniqueness of  functions $u_n$, one  can easily show that  $u_n \rp u $ weakly in $S_1^2(\H)$. For $\ka>0$, the truncation function $u_{n,\ka}: \H \ra \mathbb{R}$ defined as 
\begin{align*}
u_{n,\ka} = \left\{
\begin{array}{ll}
-\ka ,  &  \text{ if }  u_n(\xi)\leq -\ka, \\
u_n(\xi), & \text{ if } -\ka< u_n(\xi)<\ka,\\
\ka ,  &  \text{ if }  u_n(\xi)\geq \ka. \\
\end{array} 
\right.
\end{align*}
For $s\geq 2$, choose $\phi= |u_{n,\ka}|^{s-2}u_{n,\ka} \in S_1^2(\H)$ as test function for \eqref{hg23}, then 
\begin{align*}
 \frac{4(s-1)}{s^2}\int_{ \H} & |\na  (u_{n,\ka})^{s/2}|^2~ d \xi + \int_{ \H} ||u_{n,\ka}|^{s/2}|^2~ d \xi \\
&\leq  (s-1) \int_{ \H} |u_{n,\ka}|^{s-2} |\na u_{n,\ka}|^2 ~ d\xi +\int_{ \H} |u_{n,\ka}|^{s-2} u_{n,\ka}u_n ~ d\xi \\
&   = \int_{ \H}  (|\xi|^{-\la}* F_nu_n)(G_n |u_{n,\ka}|^{s-2}u_{n,\ka})~ d\xi+  \int_{ \H} (\mu-1)u |u_{n,\ka}|^{s-2}u_{n,\ka}~ d\xi. 
\end{align*}
If $2\leq s\leq \frac{2Q}{Q-\la}$ then  employing Lemma \ref{lemhg11} for $ \theta= \frac{2}{s}$, there exists a constant $C>0$ such that 
\begin{align*}
\int_{ \H}  (|\xi|^{-\la}* |F_nu_{n,\ka }|)(G_n |u_{n,\ka}|^{s-2}u_{n,\ka})~ &d\xi  \leq \int_{ \H}  (|\xi|^{-\la}* |F_n||u_{n,\ka }|)(|G_n| |u_{n,\ka}|^{s-1})~ d\xi\\
\leq& \frac{2(s-1)}{s^2}\int_{ \H}  |\na  (u_{n,\ka})^{s/2}|^2~ d \xi + \int_{ \H} ||u_{n,\ka}|^{s/2}|^2~ d \xi. 
\end{align*}
It implies for $B_{n,\ka}:= \{ \xi \in \H \;: \; |u_n(\xi)|> \ka  \}$, we deduce that 
\begin{align*}
\frac{4(s-1)}{s^2}\int_{ \H}  |\na  (u_{n,\ka})^{s/2}|^2~ d \xi  \leq& C^\prime \int_{ \H} (|u_n|^s+ |u|^s )~ d \xi \\ &+ \int_{B_{n,\ka}} (|\xi|^{-\la}* |G_n|||u_n|^{s-1})|F_nu_n|~ d\xi.
\end{align*}
Since $2\leq s\leq \frac{2Q}{Q-\la}$, using Proposition \ref{prophg1}, 
\begin{align}\label{hg24}
\int_{B_{n,\ka}} (|\xi|^{-\la}* |G_n| |u_n|^{s-1})|F_nu_n|~ d\xi \leq C \left( \int_{ \H} \bigg||G_n||u_n|^{s-1}\bigg|^p~ d \xi\right)^{\frac{1}{p}}\left( \int_{B_{n,\ka}} |F_nu_n|^r~ d \xi\right)^{\frac{1}{r}},
\end{align}
where $p$ and $r$ satisfy the relation   $\frac{1}{p}= 1+ \frac{Q-\la}{2Q}-\frac{1}{s}$ and $\frac{1}{r}=  \frac{Q-\la}{2Q}+\frac{1}{s}$. By H\"older's inequality, if $u_n \in L^s(\H)$ then $|G_n||u_n|^{s-1} \in L^p(\H)$ and $ |F_nu_n| \in L^r(\H)$, letting $\ka\ra \infty$, \eqref{hg24} yields
\begin{align*}
\lim_{\ka \ra \infty}\int_{B_{n,\ka}} (|\xi|^{-\la}* |G_n|||u_n|^{s-1})|F_nu_n|~ d\xi =0
\end{align*}
In light of inequality \eqref{hg2}, we deduce that 
\begin{align*}
\limsup_{n \ra \infty} \left( \int_{ \H} |u_n|^{\frac{sQ}{Q-2}}\right)^{\frac{Q-2}{Q}}\leq C_1 \limsup_{n \ra \infty} \int_{ \H} |u_n|^s~ d\xi.
\end{align*} 
Thus, we conclude that $ u \in L^{\frac{sQ}{Q-2}}(\H)$. By iterating over $s$  finite number of times, we can show that $  u \in L^s(\H)$ for all  $ s\in \left[2, \frac{2Q^2}{(Q-\la)(Q-2)}\right)   $. 
\QED
\end{proof}
\textbf{Proof of Theorem \ref{thmhg1}(ii)}:
	Let $F=G= |u|^{Q^*_\la-1}$, then $F,G \in L^{\frac{2Q}{Q-\la+2}}(\Om)$. With the help of
	Proposition \ref{prophg3}, $  u \in L^s(\Om)$ for all  $s\in\left[2,\frac{2Q^2}{(Q-\la)(Q-2)}\right)$.
	It implies $|u|^{Q^*_\la} \in L^r(\Om)$ for all  $ r\in \left[\frac{2(Q-2)}{2Q-\mu}, \frac{2Q^2}{(Q-\la)(2Q-\mu)}\right)$. Observe that $
	\frac{2(Q-2)}{2Q-\la}< \frac{Q}{Q-\la}<  \frac{2Q^2}{(Q-\la)(2Q-\la)}$,
	it follows that  
	\begin{align*}
	\int_{\Om} \frac{|u|^{Q^*_\la}}{|\eta^{-1}\xi|^\la} ~d\eta \in L^{\infty}(\Om).
	\end{align*}
	   Hence,	   $|-\De u | \leq C(1+|u|^{Q^*-1})$. 
	   Now adopting the technique of Brezis-Kato as presented in \cite[Lemma B.3]{struwe} (See also  \cite[Lemma 4.2]{garofalo1}) gives us
	   $u \in  L^s(\Om)$ for all $s\in [2,\infty)$. 
	Therefore, we deduce that $u \in S_{2,\text{loc}}^p(\Om)$ for all $p \geq 2$. From \cite[Theorem 5.15]{folland1}, it follows that $u \in \varGamma^\ba (\Om)$ with $\ba= 2-\frac{Q}{p}$	if $p>Q$. 
Moreover, by  \cite[Theorem 6.1]{folland1} we have  $u \in \varGamma^{2+\al}(\Om)$. Since $\Om$ is bounded domain  with $C^2$ boundary, we proceed by Moser's iteration technique as given in \cite[Chapter 8]{gilbarg}(See also \cite{lanconelli}), permits to set up the boundary regularity $ u \in\varGamma^{\al_1}(\overline{\Om})$ for some $\al_1 \in (0,1)$.     \QED

\subsection{Nonexistence result for \bm{$(P_\la)$}}
Here we first prove a Pohozaev type inequality and nonexistence result in a star-Shaped domain. First we will state a Lemma which is a consequence of a general integral identity of Rellich type. For a detailed proof, interesting readers are referred to \cite{garofalo, garofalo1}. 
\begin{Lemma}\label{lemhg13}
	Let $\Om \subset \H$ be a smooth, bounded domain with $C^2$ boundary and let $u \in \varGamma^2(\overline{\Om})$. Then 
	\begin{align*}
	2 \int_{\pa \Om} (A\nabla u\cdot N)Xu~dH_{Q-2}& - \int_{\pa \Om}|\na u|^2 X.N~ dH_{Q-2} \\ &= (2-Q) \int_{\Om} |\na u |^2 ~d\xi + 2 \int_{ \Om} Xu \De u ~d \xi,
	\end{align*}
	where $A$ is $(2N+1)\times (2N+1)$ matrix  defined in \eqref{hg25}, $N$ is the outer unit normal to $\pa \Om$, $X$ is the vector field defined as 
	\begin{equation*}
	X= \sum_{j=1}^{j=N} \left(x_j\frac{\pa}{\pa x_j}+ y_j\frac{\pa}{\pa y_j} \right)+2t\frac{\pa}{\pa t} 
	\end{equation*}
	and $dH_{Q-2}$ is the $Q-2$ dimensional Hausdorff measure in $\mathbb{R}^{2N+1} $.
\end{Lemma}
\begin{Lemma}\label{lemhg14}
	Let  $u \in  \varGamma^2(\overline{\Om}) $ be a solution of $(P_a)$ then the following holds
	\begin{equation}\label{hg31}
	\begin{aligned}
	\frac{1}{2} \int_{\pa \Om}|\na u|^2 X.N~dH_{Q-2} &+ \left(\frac{Q-2}{2}\right) \int_{\Om} |\na u |^2 ~d\xi\\
	& =  \frac{aQ}{2} \int_{\Om} |u|^2~ d \xi+ \frac{2Q-\la}{2Q^*_\la}  \int_{ \Om}\int_{\Om} \frac{|u(\xi)|^{Q^*_\la}|u(\eta)|^{Q^*_\la}}{|\eta^{-1}\xi|^{\la}}  d\eta d \xi
		\end{aligned}
	\end{equation}
\end{Lemma}
\begin{proof}
Let $\xi= (x,y,t)$ and $\eta = (x^\prime,y^\prime,t^\prime)$. Multiplying  $(P_a)$ with $Xu$ and then integrating over $\Om$ we deduce that 
\begin{align}\label{hg27}
- \int_{\Om} Xu \De u~ d\xi = a\int_{ \Om} u\; Xu ~ d \xi  + \int_{ \Om} Xu \left(	\int_{\Om} \frac{|u|^{Q^*_\la}}{|\eta^{-1}\xi|^\la} ~d\eta  \right) |u|^{Q^*_\la-2}u~ d \xi 
\end{align}
Using the fact that $\nabla u = -N |\nabla u|$ on  $\pa \Om$, it follows that 
\begin{align*}
\int_{\pa \Om} (A\nabla u\cdot N)Xu~dH_{Q-2} = \int_{\pa \Om}|\na u|^2 X.N~ dH_{Q-2}
\end{align*}
Hence employing Lemma \ref{lemhg13}, we have 
\begin{align}\label{hg28}
- \int_{\Om} Xu \De u~ d\xi = -\frac{1}{2} \int_{\pa \Om}|\na u|^2 X.N~dH_{Q-2}- \left(\frac{Q-2}{2}\right) \int_{\Om} |\na u |^2 ~d\xi
\end{align}
Taking into account  the fact that $u=0$ on $\pa \Om$, $\text{div}~X=Q$ and  divergence theorem, we obtain  
\begin{equation*}
\begin{aligned}
&\int_{ \Om} Xu(\xi)  \left(	\int_{\Om} \frac{|u(\eta)|^{Q^*_\la}}{|\eta^{-1}\xi|^\la} ~d\eta  \right) |u(\xi)|^{Q^*_\la-1}~ d \xi\\
&  = \frac{1}{Q^*_\la}\int_{ \Om} X (|u(\xi)|^{Q^*_\la}) \left(	\int_{\Om} \frac{|u(\eta)|^{Q^*_\la}}{|\eta^{-1}\xi|^\la} ~d\eta  \right)  ~ d \xi\\
&= \frac{1}{Q^*_\la}\int_{ \Om} X  \left(	\int_{\Om} \frac{|u(\eta)|^{Q^*_\la}}{|\eta^{-1}\xi|^\la} ~d\eta  |u(\xi)|^{Q^*_\la}\right)  ~ d \xi- \frac{1}{Q^*_\la}\int_{ \Om} |u(\xi)|^{Q^*_\la}  X  \left(	\int_{\Om} \frac{|u(\eta)|^{Q^*_\la}}{|\eta^{-1}\xi|^\la} ~d\eta  \right)  ~ d \xi \\
&= \frac{1}{Q^*_\la}\int_{ \pa \Om}  \left(	\int_{\Om} \frac{|u(\eta)|^{Q^*_\la}}{|\eta^{-1}\xi|^\la} ~d\eta  \right) |u(\xi)|^{Q^*_\la} X.N ~ dH_{Q-2}- \frac{1}{Q^*_\la}\int_{ \Om} \text{div}X  	\int_{\Om} \frac{|u(\eta)|^{Q^*_\la}|u(\xi)|^{Q^*_\la}}{|\eta^{-1}\xi|^\la} ~d\eta    ~ d \xi\\
& \qquad - \frac{1}{Q^*_\la}\int_{ \Om} |u(\xi)|^{Q^*_\la}  X  \left(	\int_{\Om} \frac{|u(\eta)|^{Q^*_\la}}{|\eta^{-1}\xi|^\la} ~d\eta  \right)  ~ d \xi\\
& = - \frac{Q}{Q^*_\la}\int_{ \Om}\int_{\Om} \frac{|u(\xi)|^{Q^*_\la}|u(\eta)|^{Q^*_\la}}{|\eta^{-1}\xi|^\la} ~d\eta  d \xi - \frac{1}{Q^*_\la}\int_{ \Om} |u(\xi)|^{Q^*_\la}  X  \left(	\int_{\Om} \frac{|u(\eta)|^{Q^*_\la}}{|\eta^{-1}\xi|^\la} ~d\eta  \right)  ~ d \xi\\
& = - \frac{Q}{Q^*_\la}\int_{ \Om}\int_{\Om} \frac{|u(\xi)|^{Q^*_\la}|u(\eta)|^{Q^*_\la}}{|\eta^{-1}\xi|^\la} ~d\eta  d \xi + \frac{\la}{Q^*_\la}  \int_{ \Om}\int_{\Om} \frac{|u(\xi)|^{Q^*_\la}|u(\eta)|^{Q^*_\la}}{|\eta^{-1}\xi|^{\la+4}} K_1  d\eta d \xi, 
\end{aligned}
\end{equation*}
where \begin{align*}
K_1 =&  [(x-x^\prime)^2+ (y-y^\prime)^2][\ld  x,x-x^\prime \rd + \ld y,y-y^\prime \rd] \\ &+ [(t-t^\prime)+ 2(\ld x,y^\prime\rd - \ld x^\prime, y \rd  ) ][ \ld x,y^\prime\rd - \ld x^\prime, y \rd+t ] .
\end{align*}
 Similarly for  $Y=  \sum\limits_{j=1}^{j=N} \left(x_j^\prime \frac{\pa}{\pa x_j^\prime}+ y_j^\prime\frac{\pa}{\pa y_j^\prime} \right)+2t^\prime\frac{\pa}{\pa t^\prime}$, we have 
\begin{equation*}
\begin{aligned}
\int_{ \Om} Yu(\eta) & \left(	\int_{\Om} \frac{|u(\xi)|^{Q^*_\la}}{|\eta^{-1}\xi|^\la} ~d\eta  \right) |u(\eta)|^{Q^*_\la-1}~ d \xi\\
&  = - \frac{Q}{Q^*_\la}\int_{ \Om}\int_{\Om} \frac{|u(\xi)|^{Q^*_\la}|u(\eta)|^{Q^*_\la}}{|\eta^{-1}\xi|^\la} ~d\eta  d \xi + \frac{\la}{Q^*_\la}  \int_{ \Om}\int_{\Om} \frac{|u(\xi)|^{Q^*_\la}|u(\eta)|^{Q^*_\la}}{|\eta^{-1}\xi|^{\la+4}} K_2  d\eta d \xi, 
\end{aligned}
\end{equation*}
where \begin{align*}
K_2 = & [(x-x^\prime)^2+ (y-y^\prime)^2][\ld  x^\prime,x^\prime-x \rd + \ld y^\prime,y^\prime-y \rd]\\&+ [(t-t^\prime)+ 2(\ld x,y^\prime\rd - \ld x^\prime, y \rd  ) ][ \ld x,y^\prime\rd - \ld x^\prime, y \rd-t^\prime ] 
\end{align*}
Consequently we have 
\begin{equation}\label{hg29}
\begin{aligned}
2\int_{ \Om} Xu(\xi) & \left(	\int_{\Om} \frac{|u(\eta)|^{Q^*_\la}}{|\eta^{-1}\xi|^\la} ~d\eta  \right) |u(\xi)|^{Q^*_\la-1}~ d \xi = 
 \frac{\la-2Q}{Q^*_\la}  \int_{ \Om}\int_{\Om} \frac{|u(\xi)|^{Q^*_\la}|u(\eta)|^{Q^*_\la}}{|\eta^{-1}\xi|^{\la}}  d\eta d \xi, 
\end{aligned}
\end{equation}
Moreover by divergence theorem, we get 
\begin{equation}\label{hg30}
\int_{ \Om} u\; Xu ~ d \xi= - \frac{Q}{2} \int_{\Om} |u|^2~ d \xi. 
\end{equation}
Combining \eqref{hg27} with \eqref{hg28},  \eqref{hg29} and \eqref{hg30} yields \eqref{hg31}.\QED
\end{proof}
\textbf{Proof of Theorem \ref{thmhg2}}
	Let  $u \in \varGamma^2(\overline{\Om})$ be a nontrivial solution of  $(P_a)$ then employing Lemma \ref{lemhg14}, we obtain
	\begin{equation*}
	 \int_{\pa \Om}|\na u|^2 X.N~dH_{Q-2} =  a \int_{\Om} |u|^2~ d \xi,
	\end{equation*}
	But it is given that  $a<0$ and   $\Om$ is a star shaped domain with respect ot origin, that is,  $X.N>0$, it follows that $u\equiv0$, which yields a contradiction. Hence proof follows. \QED

\end{document}